\documentclass[11pt]{article}

\usepackage{epsfig}
\usepackage{wrapfig}
\input epsf
\usepackage{here}
\usepackage{latexsym}

\usepackage{graphicx}
\usepackage[dvips]{color}
\usepackage{latexsym}
\usepackage{amsmath}
\usepackage{graphics}
\usepackage{amsthm}
\usepackage{amssymb}

  \makeatletter
  \@addtoreset{footnote}{page}
  \makeatother

\makeatletter

\advance\headheight1.2pt

\newcommand\shavedisplay{%
	\abovedisplayskip=4pt
	\belowdisplayskip=6pt
	\jot=0pt
}

\newtheorem{theorem}{Theorem}[section]
\newtheorem{cor}{Corollary}[section]
\newtheorem{proposition}{Proposition}[section]
\newtheorem{lemma}{Lemma}[section]

\newtheorem{example}{Example}[section]
\newtheorem{remark}{Remark}[section]
\newtheorem{definition}{Definition}[section]

\newtheorem{question}{Question}[section]

\newtheorem{o-problem}{Open question}[section]
\newtheorem{claim}{Claim}[section]

\begin{document}
\shavedisplay
\pagestyle{plain}

  \title{Uniqueness of stable closed non-smooth hypersurfaces with constant anisotropic mean curvature}

\author{Miyuki Koiso
\footnote
{
This work was partially supported by JSPS KAKENHI Grant Number JP18H04487. 
}
}
\date{}


  \maketitle

\begin{abstract}
We study a variational problem for piecewise-smooth hypersurfaces in the $(n+1)$-dimensional Euclidean space with an anisotropic energy. An anisotropic energy is the integral of an energy density that depends on the normal at each point over the considered hypersurface. The minimizer of such an energy among all closed hypersurfaces enclosing the same $(n+1)$-dimensional volume is unique and it is (up to rescaling) so-called the Wulff shape. The Wulff shape and equilibrium hypersurfaces of this energy for volume-preserving variations are not smooth in general. We prove that, if the anisotropic energy density function is of $C^2$ and convex, then any closed stable equilibrium hypersurface is (up to rescaling) the Wulff shape. 
We also give fundamental definitions, many examples, and generalizations of well-known concepts and formulas like Steiner's formula and Minkowski's formula to the anisotropic case. 
\end{abstract}


\noindent Contents

\S \ref{intro}. Introduction \hfill 2

\S \ref{pre}. Formulation of piecewise-$C^r$ weak immersion and anisotropic surface energy \hfill 6

\S \ref{wulff}. Definition and characterizations of the Wulff shape \hfill 7

\S \ref{Sconint}. Convex integrand \hfill 9

\S \ref{CH}. Basic properties of the Cahn-Hoffman map \hfill 10

\S \ref{CHC}. Curvatures and the regularities of the Wulff shape and the Cahn-Hoffman map \hfill 12

\S \ref{exs}. Examples \hfill 15 

\S \ref{a-curv}. Anisotropic shape operator and anisotropic curvatures \hfill 20

\S \ref{1st}. First variation formula, anisotropic mean curvature, and anisotropic Gauss map \hfill 22 

\S \ref{tube}. Anisotropic parallel hypersurface and Steiner-type formula \hfill 28

\S \ref{min}. Minkowski-type formula \hfill 30

\S \ref{s-unique}. Proof of Theorem \ref{PK2017} \hfill 33

\S \ref{PFV}. Proof of Proposition \ref{FV} \hfill 36


\S \ref{proof:ex22}. Computations for Example \ref{ex22} \hfill 39

\S \ref{pfd}. Proof of Lemma \ref{HL} (i) \hfill 40

References \hfill  41

\section{Introduction}\label{intro}

An anisotropic surface energy
was introduced by J. W. Gibbs (1839-1903) in order to model the shape of small crystals (\cite{W},\cite{Wu}),  
which is the integral of an energy density that depends on the surface normal as follows. 
Let $\gamma:S^n \rightarrow {\mathbb R}_{\ge 0}$ be a non-negative continuous function on the unit sphere $S^n=\{\nu \in {\mathbb R}^{n+1} \; | \; \|\nu\|=1\}$ in the $(n+1)$-dimensional Euclidean space ${\mathbb R}^{n+1}$. 
Let $X$ be a closed piecewise-$C^2$ weakly immersed hypersurface in ${\mathbb R}^{n+1}$ (the definition of piecewise-$C^2$ weakly immersed  hypersurface will be given in \S \ref{pre}). $X$ will be represented as a piecewise-$C^2$ mapping $X:M\to {\mathbb R}^{n+1}$ from 
an $n$-dimensional oriented connected compact $C^\infty$ manifold $M$ into ${\mathbb R}^{n+1}$. 
And the unit normal vector field $\nu$ along $X$ is defined on $M$ except a set with measure zero. Then, we can define the anisotropic energy ${\cal F}_\gamma(X)$ of $X$ as 
$
{\cal F}_\gamma(X):=\int_{M} \gamma(\nu)\:dA$, 
where $dA$ is the $n$-dimensional volume form of $M$ induced by $X$. 
If $\gamma\equiv1$, ${\mathcal F}_\gamma(X)$ is the usual $n$-dimensional volume of the hypersurface $X$. 

It is known that, if $\gamma$ is positive on $S^n$, for any positive number $V>0$, among all closed hypersurfaces in ${\mathbb R}^{n+1}$ enclosing the same $(n+1)$-dimensional volume $V$, there exists a unique (up to translation in ${\mathbb R}^{n+1}$) minimizer $W(V)$ of ${\mathcal F}_\gamma$ (\cite{T}). 
Here a closed hypersurface means that the boundary (having tangent space almost everywhere) of a set of positive Lebesgue measure. The minimizer $W(V_0)$ for 
$V_0:=(n+1)^{-1}\int_{S^n} \gamma(\nu) \;dS^n$ is called the Wulff shape (for $\gamma$) (the standard definition of the Wulff shape will be given in \S \ref{wulff}), and we will denote it by $W_\gamma$ or simply $W$. When $\gamma\equiv 1$, $W$ is the unit sphere $S^n$. All $W(V)$ are homothetic to $W$. 
$W$ is convex but not necessarily smooth. 
On the other hand, for any given convex set $\tilde{W}$ having the origin of ${\mathbb R}^{n+1}$ inside, there exists a Lipschitz continuous function $\gamma:S^n \to {\mathbb R}_{>0}$ such that the boundary $W:=\partial\tilde{W}$ of $\tilde{W}$ is the Wulff shape for $\gamma$. However, such $\gamma$ is not unique. The ``smallest'' $\gamma$ is called the convex integrand for $W$ (or, simply, convex) (for details and for another equivalent definition of the convexity, see \S \ref{wulff}).

Each equilibrium hypersurface $X$ of ${\mathcal F}_\gamma$ for variations that preserve the enclosed $(n+1)$-dimensional volume (we will call such a variation a volume-preserving variation) has constant anisotropic mean curvature and satisfies a certain condition on its ``edges'' (Proposition \ref{EL}). 
Here the anisotropic mean curvature $\Lambda$ of a piecewise-$C^2$ weakly immersed hypersurface $X$ is defined at each regular point of $X$ as (cf. \cite{R}, \cite{KP2005})
$
\Lambda:=(1/n) \large(-{\rm div}_M D\gamma+nH\gamma\large)$, 
where $D\gamma$ is the gradient of $\gamma$ on $S^n$ and $H$ is the mean curvature of $X$ (see \S \ref{1st} for details). If $\gamma\equiv 1$, then $\Lambda=H$ holds. 

We call a piecewise-$C^2$ equilibrium  hypersurface $X$ a CAMC (constant anisotropic mean curvature) hypersurface (see Definition \ref{defR} for details). 
A CAMC hypersurface is said to be stable if the second variation of the energy ${\mathcal F}_\gamma$ for any volume-preserving variation is nonnegative.

In the special case where $\gamma\equiv 1$, CAMC hypersurface is a CMC (constant mean curvature) hypersurface of $C^\omega$ class (Corollary \ref{ELI}). 
For another special $\gamma$ defined by $\gamma(\nu_1, \nu_2, \nu_3)=\sqrt{|\nu_3^2-\nu_1^2-\nu_2^2|
}$, ($\nu=(\nu_1, \nu_2, \nu_3)\in S^2$), 
${\cal F}_\gamma(X)$ is the area of the surface $X$ in ${\mathbb R}^3$ regarded as a surface in the Lorentz-Minkowski space ${\mathbb R}^3_1:=({\mathbb R}^3, dx_1^2+dx_2^2-dx_3^2)$, and 
a CAMC surface 
is a CMC surface in 
 ${\mathbb R}^3_1$ (Example \ref{reilly2}). 
These examples suggest that ${\cal F}_\gamma$ is a generalization of many important energy functionals for surfaces. 

A natural question is 

\begin{question}
Is any closed CAMC hypersurface the Wulff shape (up to translation and homothety)?
\end{question}

The answer to this uniqueness problem is not affirmative even in the case where $\gamma\equiv 1$ (\cite{HW1986}, \cite{Kap1990}, \cite{Kap1995}). 
However, it is expected that, if a closed CAMC hypersurface $X:M\to {\mathbb R}^{n+1}$ satisfies one of the following conditions (I)-(III), the image of  $X$ coincides with the Wulff shape (up to translation and homothety). 


(I) $X$ is an embedding, that is, $X$ is an injective mapping. 

(II) $X$ is stable.

(III) $n=2$ and the genus of $M$ is $0$, that is, $M$ is homeomorphic to $S^2$. 


\noindent If we assume that the Wulff shape $W_\gamma$ is a smooth strictly convex hypersurface 
(that is, $\gamma$ is of $C^2$ and the $n$ by $n$ matrix $A:=D^2\gamma+\gamma\cdot 1$ is positive definite at any point in $S^n$, where $D^2\gamma$ is the Hessian of $\gamma$ on $S^n$ and $1$ is the identity matrix of size $n$), any closed CAMC hypersurface $X$ is also smooth and the above expectation was already proved. 
In fact, if $X$ satisfies at least one of (I)-(III), it is a 
homothety of the Wulff shape.  
This fact was proved, 
under the assumption that $\gamma \in C^\infty(S^n)$ and that $A$ is positive definite, by the following papers. (I): \cite{A} for $\gamma\equiv 1$, \cite{HLMG2009} for general $\gamma$. 
(II): \cite{BD1984} for $\gamma\equiv 1$, \cite{P1998} for general $\gamma$.  
(III): \cite{H} for $\gamma\equiv 1$, \cite{KP2010} and \cite{Ando2012} for general $\gamma$.

However, the situation is not the same for more general $\gamma$ and/or $W$. 
Actually, there exists a $C^\infty$ function $\gamma:S^n\to {\mathbb R}_{>0}$ which is not a convex integrand such that there exist closed embedded CAMC hypersurfaces in ${\mathbb R}^{n+1}$ for $\gamma$ each of which is not (any homothety or translation of) the Wulff shape (\cite{JK2018}).
Also, there exists a $C^\infty$ function $\gamma:S^2\to {\mathbb R}_{>0}$ which is not a convex integrand such that there exist closed embedded CAMC surfaces in ${\mathbb R}^3$ with genus zero for $\gamma$ each of which is not (any homothety or translation of) the Wulff shape \cite{JK2018}. 

As for the uniqueness of stable closed CAMC hypersurfaces, we prove the following uniqueness result in this paper.

\begin{theorem}\label{PK2017}
Assume that $\gamma:S^n\to {\mathbb R}_{>0}$ is of $C^2$ and the convex integrand of its Wulff shape $W$. Then, the image of any closed stable piecewise-$C^2$ CAMC weakly immersed hypersurface for $\gamma$ 
whose $r$-th anisotropic mean 
curvature for $\gamma$ is integrable for $r=1, \cdots, n$ is (up to translation and homothety) $W$.
\end{theorem}
The assumption on the integrability of $r$-th anisotropic mean 
curvatures in Theorem \ref{PK2017} is a natural one. 
In fact, $r$-th anisotropic mean curvatures are quantities  that are defined by ratios of principal curvatures of $W$ and the normal curvatures of the considered hypersurface (Definition \ref{curv}), and each $r$-th anisotropic mean curvature of $W$ is $(-1)^r$ with respect to the outward-pointing normal (Remark \ref{rcurv}).  

We should remark that, although it is natural to study variational problems for anisotropic surface energy for which equilibrium surfaces have singular points, it has not yet been done sufficiently well. As for planer curves, F. Morgan \cite{M2005} proved that, if $\gamma:S^1 \to {\mathbb R}_{>0}$ is continuous and convex, then any closed equilibrium rectifiable curve for ${\mathcal F}_\gamma$ in ${\mathbb R}^2$ with area constraint is (up to translation and homothety) a covering of the Wulff shape (see \cite{MR2007} for another proof). About uniqueness of closed stable equilibria in ${\mathbb R}^3$, B. Palmer \cite{P2012} proved the same conclusion as Theorem \ref{PK2017} but under the assumptions that $\gamma:S^2 \to {\mathbb R}_{>0}$ is of $C^3$ and considered surfaces and the Wulff shape satisfy some extra assumptions (see Remark \ref{uniqueBP2} for details).

We have another important remark about the assumption of the regularity of the energy functional in Theorem \ref{PK2017}. 
If $\gamma:S^n \to {\mathbb R}_{>0}$ is a convex integrand,  
$\gamma \in C^{1, 1}$ if and only if $W_\gamma$ is uniformly convex (F. Morgan \cite{M1991})．This result implies that, if we assume that $\gamma$ is of $C^2$, then its Wulff shape $W_\gamma$ never has straight line segment (see \S \ref{Sconint} for details). This fact leads us to a future subject to weaken the assumption on the regularity of $\gamma$. 

As for the variational problem of ${\cal F}_\gamma$ without volume constraint, 
under the assumption that $\gamma \in C^2$, 
there is no closed piecewise-smooth weakly immersed hypersurface in ${\mathbb R}^{n+1}$ that is a critical point of ${\cal F}_\gamma$ (Proposition \ref {AMS}). 
This means that there is no closed piecewise-smooth weakly immersed hypersurface in ${\mathbb R}^{n+1}$ whose anisotropic mean curvature is constant zero.

We give this article as the first article that gives a systematic research of the variational problem of the anisotropic surface energy for piecewise-smooth weakly immersed hypersurfaces. Because of this reason we construct the theory from the beginning. We will not only give the proof of Theorem \ref{PK2017}, we will give fundamental definitions and many examples, prove several important results on the curvatures of the Wulff shape and the Cahn-Hoffman map, and derive important integral formulas for the critical points of our variational problem. 
This article is organized as follows. 

In \S \ref{pre}, we give the formulation of piecewise-$C^r$ weakly immersed hypersurface and the definition of its anisotropic energy. 
 
In \S \ref{wulff}, we recall  definitions of the Wulff shape and convexities of closed hypersurfaces. 
We also give the definition of the Cahn-Hoffman map $\xi_\gamma$ for $\gamma$ (Definition \ref{CHM}) as the mapping $\xi_\gamma:S^n \to {\mathbb R}^{n+1}$ defined as
$
\xi_\gamma(\nu)=D\gamma|_{\nu}+\gamma(\nu)\nu$, here the tangent space $T_\nu(S^n)$ of $S^n$ at $\nu \in S^n$ is naturally identified with the $n$-dimensional linear subspace of ${\mathbb R}^{n+1}$. It will be proved that,  
if $\gamma:S^n \to {\mathbb R}_{>0}$ is of $C^2$, then $\xi_\gamma$ is a $C^1$-(wave)front (Proposition \ref{PCH}). 
Moreover, it will be proved that, in this case the Wulff shape $W_\gamma$ is the uniquely-determined convex subset of the image $\xi_\gamma(S^n)$ of $\xi_\gamma$  including the origin in ${\mathbb R}^{n+1}$ in its interior, and that $W_\gamma=\xi_\gamma(S^n)$ if and only if $\gamma$ is convex (Theorem \ref{key2}).

In \S \ref{Sconint}, we recall the definition of convex integrand and give an analytic representation of the Wulff shape. 

In \S \ref{CH}, we give some fundamental properties of the Cahn-Hoffman map. 

In \S \ref{CHC}, we discuss curvatures and the regularities of the Wulff shape $W_\gamma$ and the Cahn-Hoffman map $\xi_\gamma$. Especially, we prove that, if the energy density function $\gamma:S^n\to {\mathbb R}_{>0}$ is of $C^2$, then the principal curvatures of $W_\gamma$ and $\xi_\gamma$ are defined at each regular point and they never vanish, and we give conditions so that they are unbounded (Theorem \ref{key1}, Corollaries \ref{cor1}, 
\ref{WCH4}
). We also prove that, even when some of the principal curvatures of them are unbounded, the integral of the mean curvature is finite (Theorem \ref{hbdd}).
 
In \S \ref{exs}, we give several examples of the energy density functions and corresponding Wulff shapes and Cahn-Hoffman maps. 

In \S \ref{a-curv}, first we define 
the Cahn-Hoffman field $\tilde{\xi}_\gamma:=\xi_\gamma\circ \nu$ along $X$ 
(Definition \ref{CHF}), which is a generalization of the Gauss map and is called also the anisotropic Gauss map of $X$ (for $\gamma$). 
Then we give the definition of the anisotropic shape operator (Definition \ref{shape}) and its representation in terms of local coordinates (Lemma \ref{shapeL}). And the definitions of anisotropic curvatures including the anisotropic mean curvature (Definition \ref{curv}) are given. We also give the representation of the anisotropic mean curvature in terms of the principal curvatures of the Cahn-Hoffman map and the corresponding normal curvatures of the considered hypersurface (Remark \ref{AMC2}). 

In \S \ref{1st}, 
we derive the Euler-Lagrange equations for our variational problem (Proposition \ref{EL}). 
In general the anisotropic Gauss map $\tilde{\xi}_\gamma=\xi_\gamma\circ \nu$ is not well-defined as a single-valued mapping because the Gauss map $\nu$ is not well-defined on the singular points of $X$. However, 
if $\gamma$ is convex and $X$ is CAMC, then strikingly $\tilde{\xi}_\gamma$ is well-defined over the whole hypersurface (Theorem \ref{CHP}).

In \S \ref{tube} and \S \ref{min}, we give the definition of anisotropic parallel hypersurface (Definition \ref{para}) and prove a Steiner-type integral formula (Theorem \ref{th-st}) and a Minkowski-type integral formula (Theorem \ref{th-min}). They are generalizations of important and useful concept (parallel hypersurfaces) and integral formulas (Steiner formula and Minkowski formula) in the isotropic case to the anisotropic case. We will use them to prove Theorem \ref{PK2017} (\S \ref{s-unique}). 

Some of the statements given in \S \ref{exs}, \S \ref{a-curv}, \S \ref{1st} will be proved in \S \ref{PFV} and in the appendices. 

\section{Formulation of piecewise-$C^r$ weakly immersed hypersurface and anisotropic surface energy} \label{pre}

First we formulate a {\it piecewise-$C^r$ weak immersion} for $r\in {\mathbb N}$. 
Let 
$M=\cup_{i=1}^k M_i$ be an $n$-dimensional oriented compact connected $C^\infty$ manifold, where each $M_i$ is an $n$-dimensional connected compact submanifold of $M$ with piecewise-$C^\infty$ boundary, and $M_i\cap M_j = \partial M_i\cap\partial M_j$, ($i, j \in \{1, \cdots, k\}$, $i\ne j$). 
We call a map 
$X:M \rightarrow{\mathbb R}^{n+1}$ a piecewise-$C^r$ weak immersion (or a piecewise-$C^r$ weakly immersed hypersurface) if $X$  satisfies the following conditions (A1), (A2), and  (A3) ($i \in \{1, \cdots, k\}$). 

(A1) $X$ is continuous, and each $X_i:=X|_{M_i} : M_i \to {\mathbb R}^{n+1}$ is of $C^r$.

(A2) If we denote by $M_i^o$ the interior of $M_i$, $X|_{M_i^o}$ is a $C^r$-immersion. 

(A3) The unit normal vector field $\nu_i:{M_i^o} \to S^n$ along $X_i|_{M_i^o}$ can be extended to a $C^{r-1}$-mapping $\nu_i:M_i \to S^n$. 
Here, if $(u^1, \cdots, u^n)$ is a local coordinate system in $M_i$, then $\{\nu_i, \partial/\partial u^1, \cdots,\partial/\partial u^n\}$ gives the canonical orientation in ${\mathbb R}^{n+1}$. 



\noindent The image $X(M)$ of a piecewise-$C^r$ weak immersion $X:M \rightarrow{\mathbb R}^{n+1}$ is also called a piecewise-$C^r$ weakly immersed hypersurface. Denote by $S(X)$ the set of all singular points of $X$, here a singular point of $X$ is a point in $M$ at which $X$ is not an immersion. 

\begin{lemma}\label{measure} 
Assume that $X:M=\cup_{i=1}^k M_i \rightarrow{\mathbb R}^{n+1}$ satisfies the above (A1), (A2), and (A3). Then, the $n$-dimensional Hausdorff measure of $X(\partial M_i)$ is zero for any $i \in \{1, \cdots, k\}$.
\end{lemma}

\noindent {\it Proof}. \ 
Let $(u^1, \cdots, u^{n-1})$ be local coordinates on $\partial M_i$. Set
$$
g_{ij}=\langle X_{u^i}, X_{u^j}\rangle, \quad i, j =1, \cdots , n-1.
$$
Then, 
the $(n-1)$-dimensional volume form of $\partial M_i$ induced by $X$ is
$$
(\det(g_{ij}))^{1/2} du^1 \wedge \cdots \wedge du^{n-1}.
$$
From this and the compactness of $\partial M_i$, the $(n-1)$-dimensional volume of $\partial M_i$ is finite. This proves the desired result. 
\hfill $\Box$

\vskip0.5truecm


Next we define the anisotropic energy of a piecewise-$C^1$ weak immersion $X:M \rightarrow{\mathbb R}^{n+1}$. 
Assume that $\gamma:S^n \rightarrow {\mathbb R}_{\ge 0}$ is a nonnegative continuous function. 
Let 
$\nu:M\setminus S(X) \to S^n$ be the unit normal vector field along $X|_{M\setminus S(X)}$. Since the $n$-dimensional Hausdorff measure of $X(S(X))$ is zero, it is reasonable to define the anisotropic energy ${\cal F}_\gamma(X)$ of $X$ as  
$$
{\cal F}_\gamma(X):=\int_{M} \gamma(\nu)\:dA.
$$ 
If $\gamma\equiv1$, ${\mathcal F}_\gamma(X)$ is the usual $n$-dimensional volume of the hypersurface $X$ 
(that is the $n$-dimensional volume of $M$ with the metric induced by $X$). 

It is often useful to consider the  homogeneous extension $\overline{\gamma}:{\mathbb R}^{n+1} \to {\mathbb R}_{\ge 0}$ of $\gamma$, which is defined as follows.
\begin{equation}\label{ext}
\overline{\gamma}(rX) := r\gamma(X), \ \forall X \in S^n, \ \forall r \ge 0.
\end{equation}
If $\overline{\gamma}$ is a convex function on ${\mathbb R}^{n+1}$  (that is, $\overline{\gamma}(v_1+v_2)\le \overline{\gamma}(v_1)+\overline{\gamma}(v_2)$ holds for all $v_1, v_2 \in {\mathbb R}^{n+1}$), $\gamma$ is called a convex integrand or simply convex (see \S \ref{Sconint} for an equivalent definition and details about convex integrand). 

\section{Definition and  characterizations of the Wulff shape}\label{wulff}

In this section, we assume that $\gamma:S^n \rightarrow {\mathbb R}_{>0}$ is a positive continuous function. We give the definition of the Wulff shape and some of its characterizations which we will need in the following sections (see \cite{T} for details). 
A positive continuous function on $S^n$ is sometimes called an integrand. 

The boundary $W_\gamma$ of the convex set 
$
\tilde{W}[\gamma] :=\cap _{\nu \in S^n} 
\bigl\{
X \in {\mathbb R}^{n+1} |\langle X, \nu\rangle \le \gamma(\nu)
\bigr\}
$ 
is called the Wulff shape for $\gamma$, where $\langle \ , \ \rangle$ means the standard inner product in ${\mathbb R}^{n+1}$. ($\tilde{W}[\gamma]$ itself is often called the Wulff shape.) 
In the special case where $\gamma\equiv 1$, $W_\gamma$ coincides with $S^n$.

If the homogeneous extension $\overline{\gamma}:{\mathbb R}^{n+1} \to {\mathbb R}_{\ge 0}$ of $\gamma$ is convex and satisfies $\overline{\gamma}(-X)=\overline{\gamma}(X)$, ($\forall X \in {\mathbb R}^{n+1}$), $\overline{\gamma}$ defines a norm of ${\mathbb R}^{n+1}$. 
The Wulff shape $W_\gamma$ coincides with the unit sphere 
$$
\{Y \in {\mathbb R}^{n+1} \:|\: \gamma^{\ast}(Y)=1\}
$$
of the dual norm $\gamma^{\ast}$ of $\overline{\gamma}$ which is defined as  
$$
\gamma^{\ast}(Y)=\sup \{Y\cdot Z\:|\:\gamma(Z)\le 1\}, \quad \forall Y \in {\mathbb R}^{n+1}.
$$ 

$W_\gamma$ is not smooth in general. $W_\gamma$ is smooth and strictly convex 
(that is, 
each principal curvature of $W_\gamma$ with respect to the inward-pointing normal is positive at each point of $W_\gamma$
) 
if and only if   $\gamma$ is of $C^2$ and the $n$ by $n$ matrix $D^2\gamma+\gamma\cdot 1$ is positive definite at any point in $S^n$, where $D^2\gamma$ is the Hessian of $\gamma$ on $S^n$ and $1$ is the identity matrix of size $n$ (cf. Corollary \ref{WCH4}). 
When $\gamma$ satisfies such a convexity condition, the functional ${\cal F}_\gamma$ is sometimes called a  
(constant coefficient) parametric elliptic functional (\cite{Fe}, \cite{GT}).

\begin{definition}\label{CHM}
Assume that $\gamma:S^n \to {\mathbb R}_{\ge 0}$ is of $C^1$. We call the continuous map $\xi:S^n \to {\mathbb R}^{n+1}$ defined as
\begin{equation}\label{CHE}
\xi(\nu) :=\xi_\gamma(\nu):=D\gamma|_{\nu}+\gamma(\nu)\nu, \quad \nu \in S^n
\end{equation}
the Cahn-Hoffman map (for $\gamma$). 
Here the tangent space $T_\nu(S^n)$ of $S^n$ at $\nu \in S^n$ is naturally identified with the $n$-dimensional linear subspace of ${\mathbb R}^{n+1}$. 
\end{definition}

\begin{remark}\label{homoge1}
It is easy to show that the Cahn-Hoffman map $\xi_\gamma:S^n \to {\mathbb R}^{n+1}$ is represented by using $\overline\gamma$ as
\begin{equation}\label{homoge2}
\xi_\gamma(\nu)=\overline D\overline\gamma|_{\nu}, \quad \nu \in S^n,
\end{equation}
where $\overline D$ is the gradient in ${\mathbb R}^{n+1}$.
\end{remark}

If $\gamma:S^n \to {\mathbb R}_{>0}$ is of $C^1$, 
the Wulff shape $W_\gamma$ is a subset of the image $\hat{W}_\gamma:=\xi_\gamma(S^n)$ of $\xi_\gamma$, and  
$\hat{W}_\gamma=W_\gamma$ holds if and only if $\gamma$ is convex (Theorem \ref{key2}). 

\begin{definition}\label{supp}{\rm 
Assume that $S$ is a closed hypersurface in ${\mathbb R}^{n+1}$ that is the boundary of a bounded open set $\Omega$. 
Denote by $\overline{\Omega}$ the closure of $\Omega$, that is, $\overline{\Omega}=\Omega \cup S$.

(i) $S$ is said to be strictly convex if, for any straight line segment $PQ$ connecting two distinct points $P$ and $Q$ in $S$, $PQ \subset \overline{\Omega}$ and $PQ\cap S=\{P, Q\}$ hold.

(ii) $S$ is said to be convex if, for any straight line segment $PQ$ connecting two points $P$ and $Q$ in $S$, $PQ \subset \overline{\Omega}$ holds.

(iii) (cf. \cite{R1970}) Assume that $S$ is convex. The support function of $S$ is a mapping $\sigma_S:S^n\to {\mathbb R}$ which is defined as follows. For $\nu \in S^n$ and $a\in{\mathbb R}$, define a closed half space $\Pi(\nu, a)$ as 
$$
\Pi(\nu, a):=\bigl\{
X \in {\mathbb R}^{n+1} \;|\;\langle X, \nu\rangle \le a
\bigr\}.
$$
Then, set
$$
\sigma_S(\nu):=\min \{a \in {\mathbb R}\;|\;\Pi(\nu, a)\supset S\}.
$$
}\end{definition}

Now assume that $\gamma$ is of $C^2$ and convex. Then, for any $\nu \in S^n$, the outward-pointing unit normal to $W_\gamma$ at the point $\xi(\nu)$ is well-defined and it coincides with $\nu$ (Proposition \ref{PCH}). Hence  
$\gamma$ is the support function of $W_\gamma$, because 
$\langle\xi(\nu), \nu \rangle
=\langle D\gamma|_\nu+\gamma(\nu)\nu, \nu \rangle
=\gamma(\nu)$ holds.

Conversely, consider an arbitrary piecewise-$C^1$ strictly convex closed hypersurface $W$ in ${\mathbb R}^{n+1}$. Assume that the origin of ${\mathbb R}^{n+1}$ is contained in the open domain bounded by $W$. 
Denote by $\gamma$ the support function of $W$. Then, $W$ is the Wulff shape for $\gamma$ (cf. \cite{T}). 

Therefore, there is a one-to-one correspondence between the following two sets $A$ and $B$. 
\begin{eqnarray}
A\!\!&=&\!\!\{
W \subset {\mathbb R}^{n+1} \;|\;
\textup{$W$ is a piecewise-$C^1$ strictly convex closed hypersurface, and} \nonumber\\  
&&\textup{the origin of ${\mathbb R}^{n+1}$ is contained in the open domain bounded by W.} 
\}, \label{defA}\\
B\!\!&=&\!\!\{
\gamma:S^n \to {\mathbb R}_{>0} \;|\;
\textup{
$\gamma$ is of $C^2$ and convex}.
\}. \label{defB}
\end{eqnarray}
The correspondence is given by the bijection $F:A \to B$ which maps $W \in A$ to the support function of $W$. The inverse mapping $F^{-1}:B \to A$ maps $\gamma \in B$ to its Wulff shape $W_\gamma$.  

An integrand $\gamma$ is said to be crystalline if its Wulff sape $W_\gamma$ is a polyhedral hypersurface (\cite{T}). 

%
%

\section{Convex integrand}\label{Sconint}

In this section, we prove that the Cahn-Hoffman map $\xi_\gamma$ (Definition \ref{CHM}) gives a representation of the Wulff shape $W_\gamma$ if $\gamma$ is of $C^2$ and convex (Theorem \ref{key2}). 

\begin{definition}{\rm
For any nonnegative continuous function $\gamma:S^n\to{\mathbb R}_{\ge 0}$, the set $\{\gamma(\nu)\nu \;;\; \nu \in S^n\}$ is called  the Wulff plot for $\gamma$ (or the radial plot of $\gamma$).
}
\end{definition}

\begin{definition}[\cite{T}]\label{conint}
A positive continuous function $\gamma:S^n \to {\mathbb R}_{>0}$ is called a convex integrand if the Wulff plot of the function
$$
1/\gamma :S^n \to {\mathbb R}_{>0}, \quad (1/\gamma)(\nu):=\gamma(\nu)^{-1}, \ \forall \nu\in S^n,
$$
is convex. 
\end{definition}

\begin{remark}[\cite{T}]\label{conr}
(i) For any continuous $\gamma:S^n \to {\mathbb R}_{>0}$, there exists a unique convex integrand $\tilde{\gamma}$ such that $W_{\gamma}=W_{\tilde{\gamma}}$ holds. 

(ii) $\tilde{\gamma}$ is the smallest integrand among all integrands having the same Wulff shape, that is
$$
\tilde{\gamma}(\nu)=\min\{
f(\nu)\;|\; f\in C^0(S^n, {\mathbb R}_{>0}), W_f=W_{\tilde{\gamma}}
\}, \quad \forall \nu\in S^n
$$
holds. 
\end{remark}

Recall that $\overline\gamma : {\mathbb R}^{n+1} \to {\mathbb R}_{\ge0}$ is the homogeneous extension of $\gamma : S^n \to {\mathbb R}_{\ge 0}$ (\S \ref{pre}).

\begin{remark}[cf. \S 3 of \cite{M1991}]\label{morgan1991} {\rm 
Assume that $\gamma:S^n\to {\mathbb R}_{>0}$ is continuous. 
Then $\gamma$ is convex if and only if $\overline{\gamma}$ is a convex function. 
}\end{remark}


%
%
%
%
%
%
%


Although maybe the following result has been already proved somewhere, we will give its proof for completeness. 

\begin{theorem}\label{key2} 
Assume that $\gamma:S^n\to {\mathbb R}_{>0}$ is of $C^2$. 

(i) The Wulff shape $W_\gamma$ for $\gamma$ is a subset of the image $\hat{W}_\gamma:=\xi_\gamma(S^n)$ of the Cahn-Hoffman map $\xi_\gamma:S^n\to {\mathbb R}^{n+1}$.

(ii) The following (a), (b) and (c) are equivalent.

\hspace{5mm} (a) $\gamma$ is a convex integrand. 

\hspace{5mm} (b) $W_\gamma=\hat{W}_\gamma$ holds.

\hspace{5mm} (c) $D^2\gamma+\gamma\cdot 1$ is positive-semidefinite, that is, the eigenvalues are all nonnegative on the tangent space at each point in $S^n$.

(iii) If $\gamma$ is uniformly convex, that is, the eigenvalues of $D^2\gamma+\gamma\cdot 1$ are all positive on the tangent space at each point in $S^n$, then $\xi_\gamma:S^n\to {\mathbb R}^{n+1}$ is an embedding onto 
$W_\gamma$. In this case, for each $\nu \in S^n$, 
$\nu$ coincides with the outward-pointing unit normal of $W_\gamma$ at the point $\xi_\gamma(\nu)$, and 
$\gamma$ is the support function of $W_\gamma$, that is, 
$\gamma(\nu)=\langle\xi_\gamma(\nu), \nu \rangle$ holds.
\end{theorem}

\noindent {\it Proof.}\ 
First we prove (i). 
The support function of $W_\gamma$ in the sence of Proposition \ref{support} coincides with $\gamma$ almost everywhere on $W_\gamma$. On the other hand, the support function of $\xi_\gamma$ is exactly $\gamma:S^n \to {\mathbb R}_{>0}$. Since a support function determines the hypersurface uniquely, $W_\gamma\subset \xi_\gamma(S^n)$ holds. 

Now we prove (ii). First we prove that (a) and (c) are equivalent. 
Note that $\gamma$ is convex
(that is, $\overline{\gamma}(v_1+v_2)\le \overline{\gamma}(v_1)+\overline{\gamma}(v_2)$ holds for all $v_1, v_2 \in {\mathbb R}^{n+1}$)
 if and only if the Hessian $Hess(\overline\gamma)$ of $\overline\gamma$ is positive-semidefinite. 
It is sufficient to prove that $D^2\gamma+\gamma\cdot 1$ is positive-semidefinite at $\nu=(0, \cdots, 0, 1)$. 
Then, we can use $(x_1, \cdots, x_n)$ as local coordinate of $S^n$ around $\nu$, 
where $(x_1, \cdots, x_{n+1})$ is the canonical coordinate in ${\mathbb R}^{n+1}$. Then, 
$$
D^2\gamma=
\sum_{i,j=1}^n\Biggl(
\frac{\partial^2\gamma}{\partial x_i \partial x_j}
-\sum_{k=1}^n\Gamma^k_{ij}\frac{\partial\gamma}{\partial x_k}
\Biggr)dx_i\otimes dx_j,
$$
where $\Gamma^k_{ij}$ is the Christoffel symbols. 
Since $\nu$ gives the normal direction of $\hat{W}$ at $\xi(\nu)$ (Proposition \ref{PCH}), and since $\overline\gamma$ is homogeneous, we can show that
$$
D^2\gamma+\gamma\cdot 1=Hess(\overline\gamma)\big|_{S^n}
$$
holds. This implies that (a) and (c) are equivalent. 

Next we prove that (a) and (b) are equivalent. 
$\gamma$ is convex if and only if $\gamma$ is the support function of $W_\gamma$ (\cite{T}). 
On the other hand, $\gamma$ is the support function of $\xi_\gamma$ in the sence of Proposition \ref{support}. 
The property that the support functions of $W_\gamma$ and $\xi_\gamma$ coincide is equivalent to the property that $W_\gamma$ and $\xi_\gamma(S^n)$ coincide. 

Lastly we prove (iii). 
The first statement follows from the fact that $d\xi_\gamma=D^2\gamma+\gamma\cdot 1$ and (ii). The second statement follows from Proposition \ref{support}. 
%
\hfill $\Box$


\section{Basic properties of the Cahn-Hoffman map}\label{CH}

In Theorem \ref{key2}, we observed that, if $\gamma:S^n \to {\mathbb R}_{>0}$ is of $C^2$ and convex, then the Cahn-Hoffman map $\xi_\gamma$ gives a parametrization of the Wulff shape $W_\gamma$. 
In this section we also assume that $\gamma$ is of $C^2$, but we do not assume that $\gamma$ is convex, and we discuss geometry of $\xi_\gamma$, especially its singular points and its curvatures. 

First we give some fundamental properties of $\xi_\gamma$. 

\begin{proposition}\label{PCH}
Assume that $\gamma:S^n \to {\mathbb R}_{>0}$ is of $C^2$. Then, the Cahn-Hoffman map $\xi:=\xi_\gamma$ satisfies the following (i) and (ii), hence $\xi$ is a $C^1$-(wave)front.

(i) 
\begin{equation}\label{fr1}
\langle(d\xi)_{\nu}(u), \nu\rangle=0, \quad \forall \nu \in S^n, \ \forall u\in T_{\nu}S^n.
\end{equation}

(ii) The mapping 
\begin{equation}\label{fr2}
(\xi, id_{S^n}):S^n \to {\mathbb R}^{n+1}\times S^n, \quad 
(\xi, id_{S^n})(\nu):=(\xi(\nu), \nu)
\end{equation}
is a $C^1$-immersion.
\end{proposition}

We will give the proof of Proposition \ref{PCH} at the end of this section for completeness.

Proposition \ref{PCH} (i) yields the following Definition \ref{CCH1} and Corollary \ref{CCH2}.

\begin{definition}\label{CCH1}
Assume that $\gamma:S^n \to {\mathbb R}_{>0}$ is of $C^2$. Then, at any point $\nu \in S^n$ we call the hyperplane perpendicular to $\nu$ the tangent hyperplane of $\xi_\gamma$ at $\nu$ (or at $\xi_\gamma(\nu))$.
\end{definition}

\begin{cor}\label{CCH2}
Assume that $\gamma:S^n \to {\mathbb R}_{>0}$ is of $C^2$. Then, at any point $\nu \in S^n$ where $\xi_\gamma$ is an immersion, $\nu$ itself gives a unit normal to $\xi_\gamma$.
\end{cor}

The following proposition is an immediate consequence of $\langle \xi_\gamma(\nu), \nu\rangle=\gamma(\nu)$, which  gives an important relation between $\gamma$ and its Cahn-Hoffman map $\xi_\gamma$.

\begin{proposition}\label{support}
Assume that $\gamma:S^n \to {\mathbb R}_{>0}$ is of $C^2$. Then, $\gamma$ is the support function of $\xi_\gamma$, that is, $\gamma(\nu)$ is the distance between the origin of ${\mathbb R}^{n+1}$ and the tangent hyperplane of $\xi_\gamma$ at the point $\xi_\gamma(\nu)$.
\end{proposition}

\noindent{\it Proof of Proposition \ref{PCH}.} \ 
(ii) is obvious because $id_{S^n}:S^n\to S^n$ is an immersion. 

We will prove (i). 
Let $(V, \varphi)=(V, (y^1, \cdots, y^n))$ be a local coordinate neighborhood of $S^n$. 
Represent $\xi=(\xi^1, \cdots, \xi^{n+1})$. Then
\begin{equation}\label{dxi}
(d\xi)_{\nu}\Bigl(\frac{\partial}{\partial y^j}\Bigr)_{\nu} = \frac{\partial\xi\circ\varphi^{-1}}{\partial y^j}.
\end{equation}

Set $V':=\varphi(V)$ and 
$$
\psi:=\varphi^{-1}:V' \to V.
$$
The standard Riemannian metric 
$\displaystyle ds^2=\sum_{i,j=1}^ng_{ij}dy^idy^j
$ in $V$ is given by
$$
g_{ij}:=\langle \psi_{y^i}, \psi_{y^j}\rangle.
$$
Set $(g^{ij}):=(g_{ij})^{-1}$, and using the identification 
$$
\frac{\partial}{\partial y^j}=\psi_{y^j},
$$
the gradient $D\gamma$ of $\gamma$ is given as
$$
D\gamma=\sum_{i,j=1}^n g^{ij}\gamma_{y^i}\psi_{y^j}, \quad \gamma_{y^i}:=\frac{\partial(\gamma\circ\varphi^{-1})}{\partial y^i}. 
$$
Therefore, 
$$
\xi=D\gamma+\gamma(\nu)\nu
=\sum_{i,j=1}^n g^{ij}\gamma_{y^i}\psi_{y^j} + \gamma(\nu)\nu,
$$
and hence
\begin{eqnarray}
\frac{\partial\xi}{\partial y^k}
&=&
\sum_{i, j}(g^{ij}\gamma_{y^i}\psi_{y^j})_{y^k}
+ (\gamma(\nu))_{y^k}\nu + \gamma(\nu)\nu_{y^k} \nonumber\\
&=&
\sum_{i, j}(g^{ij}\gamma_{y^i})_{y^k}\psi_{y^j}
+\sum_{i, j}g^{ij}\gamma_{y^i}\psi_{y^jy^k}
+ (\gamma(\nu))_{y^k}\nu + \gamma(\nu)\nu_{y^k} \nonumber\\
&=:& I+II+III+IV. \label{1234}
\end{eqnarray}
It is clear that 
\begin{equation}\label{tan}
\langle I, \nu\rangle=0, \quad \langle IV, \nu\rangle=0
\end{equation}
holds. 
We will compute the normal component of II. 
Differentiate $\langle \nu, \psi_{y^j}\rangle=0$ to get
$$
\langle \nu_{y^k}, \psi_{y^j}\rangle + \langle \nu, \psi_{y^jy^k}\rangle=0.
$$
Hence, 
$$
\langle \nu, \psi_{y^jy^k}\rangle=-\langle \nu_{y^k}, \psi_{y^j}\rangle=h_{kj},
$$
here $\sum_{i, j}h_{ij}dy^idy^j$ is the second fundamental form on $S^n$. This gives
\begin{equation}\label{A1}
B:=\langle\sum_{i, j}g^{ij}\gamma_{y^i}\psi_{y^jy^k}, \nu\rangle
= \sum_{i, j}g^{ij}h_{kj}\gamma_{y^i}.
\end{equation}
In order to prove (\ref{fr1}), let $\nu^\ast$ be the antipodal point of $\nu$, and let $\pi:S^n\setminus\{\nu^\ast\}\to {\mathbb R}^n$ be the stereographic projection from $\nu^\ast$. 
Then, at $\nu$, 
\begin{equation}\label{st}
(g_{ij})=(\delta_{ij}), \quad (h_{ij})=-(\delta_{ij})
\end{equation}
holds. Therefore we have
\begin{equation}\label{A2}
B=-\sum_{i,j}\delta^{ij}\delta_{kj}\gamma_{y^i}=-\gamma_{y^k}.
\end{equation}
By (\ref{dxi}), (\ref{1234}), (\ref{tan}), and (\ref{A2}), we obtain
$$
\langle(d\xi)_{\nu}(u), \nu\rangle=0, \quad \forall u\in T_{\nu}S^n,$$
which proves (i). 
\hfill$\Box$

\section{Curvatures and the regularities of the Wulff shape and the Cahn-Hoffman map}\label{CHC}

The curvatures and the regularities of the Wulff shape and the Cahn-Hoffman map depend on the anisotropic energy density function $\gamma$, which is the subject of this section. We will explain why we assume that $\gamma$ is of $C^2$, not of $C^3$ in our main results. Also we will observe that the assumption $\gamma \in C^2$ is too strong in a certain sence. 

First we recall the following result. 

\begin{remark}\label{MNM} {\rm 
Assume that $\gamma:S^n \to {\mathbb R}_{>0}$ is a convex integrand. 

(i) (F. Morgan \cite{M1991}) $\gamma \in C^{1, 1}$ if and only if $W_\gamma$ is uniformly convex．

(ii) (H. Han and T. Nishimura \cite{HN2017}) $\gamma \in C^1$ if and only if $W_\gamma$ is strictly convex．
}\end{remark}

This result implies that if a convex integrand $\gamma:S^n \to {\mathbb R}_{>0}$ is of $C^1$, then $W_\gamma$ does not include any straight line segment. 
In this case, $\gamma$ is not crystalline, that is, the Wulff sape $W_\gamma$ is not a polyhedral hypersurface. In this sence, $\gamma \in C^1$ is a too strong assumption. 
Actually, if a convex integrand $\gamma:S^n \to {\mathbb R}_{>0}$ is differentiable at all points in $S^n$, then it is of $C^1$ (\cite{R1970}). Therefore, we have:

\begin{remark}\label{non-d} {\rm 
Assume that $\gamma:S^n \to {\mathbb R}_{>0}$ is a convex integrand. Then, $W_\gamma$ includes a straight line segment $\ell$ if and only if $\gamma$ is not differentiable at a certain point $\nu\in S^n$. More precisely, $\nu$ is normal to $\ell$ and $\gamma$ is not partially differentiable at $\nu$ in the direction of $\ell$.
}\end{remark}

Next we consider the case where $\gamma$ is of $C^2$. 

\begin{theorem}\label{key1}
If $\gamma:S^n\to {\mathbb R}_{>0}$ is of $C^2$, then the following (i) and (ii) hold. 

(i) The principal curvatures at any regular point of the Cahn-Hoffman map $\xi_\gamma$ never vanish. 

(ii) For any singular point $\nu \in S^n$ of $\xi_\gamma$, and for any smooth one-parameter family $\nu_t \in S^n$ with $\displaystyle \lim_{t\to \infty} \nu_t=\nu$ of regular points of $\xi_\gamma$ with principal curvatures $\mu_1(t), \cdots, \mu_n(t)$, the limit $\displaystyle \lim_{t\to \infty} |\mu_i(t)|$ exists and it is either $\infty$ or a nonzero real value, ($i=1, \cdots, n$). Moreover, near any singular point $\nu$ of $\xi_\gamma$, 
at least one of the principal curvatures of $\xi_\gamma$ is unbounded．
\end{theorem}

\noindent {\it Proof.} \ 
(i) Note that the operator $D^2\gamma+\gamma\cdot 1$ is symmetric on $S^n$. 
Recall that $\xi_\gamma:S^n \to {\mathbb R}^{n+1}$ is defined by $\xi(\nu)=D\gamma|_\nu +\gamma(\nu)\nu$. Since, at any regular point $\nu \in S^n$ of $\xi_\gamma$, $\nu$ itself gives a unit normal to $\xi_\gamma$ (Corollary \ref{CCH2}), $\xi_\gamma$ gives the inverse of a unit normal vector field along $\xi_\gamma$ itself near $\nu$. 
Hence, $A:=d\xi_\gamma=D^2\gamma+\gamma\cdot 1$ is the inverse of the differential of the unit normal vector field along $\xi_\gamma$, which implies that the  
eigenvalues $\rho_1, \cdots, \rho_n$ of $A$ are the negatives of the reciprocals of the principal curvatures $\mu_1, \cdots, \mu_n$ 
of $\xi_\gamma$ with respect to the unit normal $\nu$. 
This means that $\mu_j=1/\rho_j\ne 0$ for all $j \in \{1, \cdots, n\}$ at any point in $S^n$. 

(ii) If $\nu \in S^n$ is a singular point of $\xi_\gamma$, the matrix $A$ has at least one zero eigenvalue. This fact with the observation above gives the desired result. 
\hfill$\Box$

\vskip0.5truecm

Example \ref{ex4'} gives a good example for Theorem \ref{key1} (ii).


Recall that, if $\gamma:S^n \to {\mathbb R}_{>0}$ is of $C^2$ and convex, then the Cahn-Hoffman map $\xi_\gamma$ gives a parametrization of the Wulff shape $W_\gamma$ (Theorem \ref{key2}). This fact with  Theorem \ref{key1} (ii) gives the following results.

\begin{cor}\label{cor1}
Assume that $\gamma:S^n \to {\mathbb R}_{>0}$ is convex. 
If $\gamma$ is of $C^2$, then 
near any singular point of $W_\gamma$, 
at least one of the principal curvatures of $W_\gamma$ is unbounded．
\end{cor}
Examples \ref{ex22}, \ref{ex5} give good examples for Corollary \ref{cor1}. 

The following known result is an immediate consequence of Theorem \ref{key1} (i).

\begin{cor}\label{WCH4}
For $\gamma \in C^2(S^n, {\mathbb R}_{>0})$, 
the following (i) and (ii) are equivalent.

(i) $W_\gamma$ is a closed strictly-convex smooth hypersurface, that is, all of the principal curvatures of $W$ are positive for the inward-pointing unit normal.

(iii) $D^2\gamma+\gamma\cdot 1$ is positive-definite, that is, the eigenvalues are all positive, on the tangent space at each point in $S^n$.
\end{cor}

Even when some of the principal curvatures of $\xi_\gamma$ are unbounded, the integral of the mean curvature of $\xi_\gamma$ is finite as follows. 

\begin{theorem}\label{hbdd}
Assume that $\gamma:S^n \to {\mathbb R}_{>0}$ is of $C^2$. Then the mean curvature $H$ of the Cahn-Hoffman map $\xi_\gamma$ is defined on $S^n\setminus S(\xi_\gamma)$, and  its improper integral $\displaystyle \int_{S^n\setminus S(\xi_\gamma)} H \;dA$ converges.
\end{theorem}

\noindent{\it Proof.} \ 
Take any point $\nu \in S^n$. 
Because $d\xi_\gamma=D^2\gamma+\gamma\cdot 1$ is symmetric, 
we can take a locally defined
frame $\{e_1, \cdots, e_n\}$ on $S^{n}$ such that $(D^2\gamma+\gamma\cdot 1)e_i=\rho_ie_i$ holds at $\nu$, where $\rho_i$ are the eigenvalues of $d\xi_\gamma$. Because of Proposition \ref{PCH} (i), the basis
$\{e_1, \cdots, e_n\}$ at $\nu$ also serves as an orthogonal basis for the tangent hyperplane of $\xi_\gamma$ at $\nu$. 
If $\nu$ is a regular point of $\xi_\gamma$, as the proof of Theorem \ref{key1} (i), $\rho_1, \cdots, \rho_n$ are the negatives of the reciprocals of the principal curvatures $\mu_1, \cdots, \mu_n$ 
of $\xi_\gamma$ with respect to the unit normal $\nu$. 
Hence
$$
d\xi_\gamma=(D^2\gamma+\gamma\cdot  1)= \left(\begin{array}{llllll}
1/\mu_1 \quad 0 \qquad  \quad \cdots \quad \ \ \quad 0 \\
\  0 \qquad 1/\mu_2 \ \  0 \ \  \quad \cdots \quad \ \  0 \\
\ 0 \qquad 0 \qquad \cdot \ \   0\quad \cdots \quad 0\\
\  \cdot   \quad \quad \cdots \ \  \qquad \cdot \ \quad \cdots \quad \cdot\\
\ 0  \qquad \quad \cdots \quad  \quad \ 0 \quad \cdot \quad 0\\
 \ 0 \ \qquad 0 \quad \  \quad \cdots \qquad 0 \ \    1/\mu_n
\end{array} \right). $$
holds. Therefore  we have
\begin{equation}\label{hda}
nH \;dA=(\mu_1+\cdots +\mu_n)\frac{1}{|\mu_1\cdots\mu_n|}\;du_1\wedge\cdots \wedge du_n,
\end{equation}
where $(u_1, \cdots, u_n)$ is the corresponding local coordinate in $S^n$. Here,
\begin{eqnarray}
|\mu_1+\cdots +\mu_n|\frac{1}{|\mu_1\cdots\mu_n|}
&\le& 
\frac{|\mu_1|}{|\mu_1\cdots\mu_n|}
+
\cdots
+
\frac{|\mu_n|}{|\mu_1\cdots\mu_n|} \nonumber\\
&=&
|\rho_2\cdots\rho_n|+ \cdots +|\rho_1\cdots\rho_{n-1}| \label{hda2}
\end{eqnarray}
holds. Since $\gamma$ is of $C^2$, the right hand side of (\ref{hda2}) is bounded, which implies the desired result. 
\hfill$\Box$

\section{Examples}\label{exs}

In this section, we give several examples of integrands $\gamma:S^n \to {\mathbb R}_{>0}$ with various regularities, and we show their Wulff shapes and Cahn-Hoffmann maps. 
Sometimes we will assume $n=1$ or $n=2$ for simplicity. However, it is easy to generalize them to higher dimensional examples (for example, by rotation). Below we use the notation $\nu=(\nu_1, \cdots, \nu_{n+1})$ for $\nu \in S^n$.

\begin{example}[cf.\cite{R}. $\gamma \in C^\infty$, uniformly convex]\label{reilly} {\rm 
Set $n=2$. Let $a_1, a_2, a_3$ be positive constants. Set
$$
\gamma(\nu):=\sqrt{a_1^2\nu_1^2+a_2^2\nu^2_2+a_3^2\nu_3^2}.
$$
Then, $\gamma \in C^\infty$ and it is uniformly convex. $W_\gamma$ is the ellipsoid
$$\frac{x_1^2}{a_1^2}+\frac{x_2^2}{a_2^2}+\frac{x_3^2}{a_3^2}=1.
$$
Since, by the transformation 
$$x_1'=x_1/a_1, \ x_2'=x_2/a_2, \ x_3'=x_2/a_3,
$$
the functional
$${\cal F}_\gamma=\int \sqrt{a_1^2\nu_1^2+a_2^2\nu^2_2+a_3^2\nu_3^2 }\:dA
$$
becomes $a_1a_2a_3$ times the usual area,
an immersion $X=(x_1, x_2, x_3)$ has constant anisotropic mean curvature if and only if $(x_1', x_2', x_3')$ has constant mean curvature.
}\end{example}

\begin{example}[\cite{HKT2013}. $\gamma:S^n\to {\mathbb R}_{\ge 0}$, $\gamma \in C^0$, $\gamma \notin C^1$, $\gamma$ is not convex]\label{reilly2} {\rm
Set $n=2$. 
Define $\gamma\!:\!S^2\to {\mathbb R}_{\ge 0}$ as 
$$\gamma(\nu):=\!\sqrt{|\nu_3^2-\nu_1^2-\nu_2^2|}.$$
Then, ${\cal F}_\gamma(X)$ is the area of the surface $X$ regarded as a surface in the Lorentz-Minkowski space ${\mathbb R}^3_1:=({\mathbb R}^3, dx_1^2+dx_2^2-dx_3^2)$.
Hence, a CAMC surface
is 
 a CMC surface in 
 ${\mathbb R}^3_1$. 
For a graph $z=f(x, y)$, 
$\Lambda$ and the mean curvature $H_L$ as a surface in ${\mathbb R}^3_1$ satisfy
$$
\Lambda=H_L=(1/2)|1-f_x^2-f_y^2|^{-3/2}\Bigl[(1-f_y^2)f_{xx}+2f_xf_yf_{xy}+(1-f_x^2)f_{yy}\Bigr].
$$
Hence, the equation $\Lambda=$ constant is 

(i) elliptic on ``space-like parts'' where $1-f_x^2-f_y^2>0$ holds.

(ii) hyperbolic on ``time-like parts'' where $1-f_x^2-f_y^2<0$ holds.
\newline 
The image of the Cahn-Hoffman map is the elliptic hyperboloid of two sheets (one sheet) for the space-like (time-like) part.

    \begin{figure}[H]
  \centering   \includegraphics[width=30mm,height=30mm,angle=0]{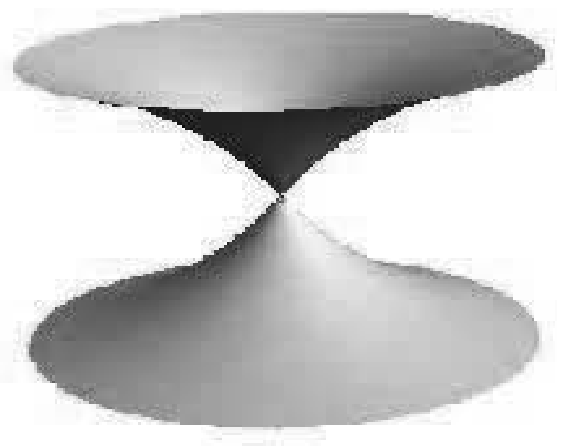}
\hspace{2mm}
\includegraphics[width=50mm,height=30mm,angle=0]{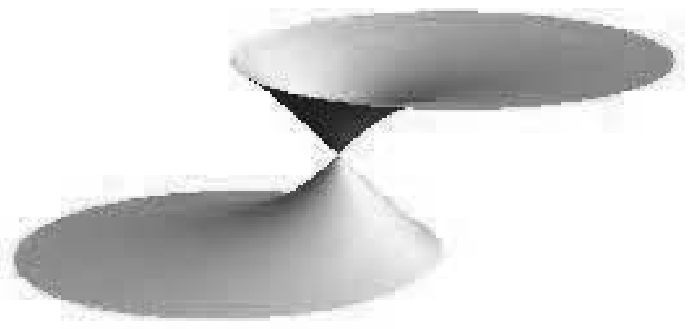}
\hspace{2mm}
\includegraphics[width=20mm,height=25mm,angle=0]{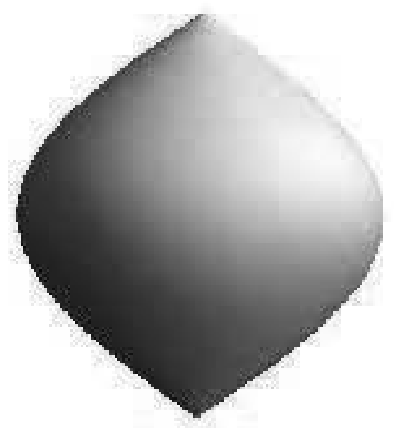}
\hspace{5mm}
\includegraphics[width=18mm,height=25mm,angle=0]{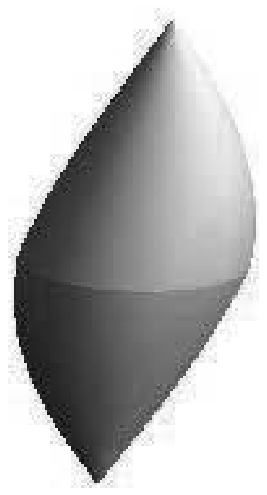}
\caption{Surfaces with zero mean curvature foliated by circles in ${\mathbb R}^3_1$. From the left, the space-like catenoid, the space-like Riemann-type maximal surface, the time-like catenoid, and the time-like Riemann-type minimal surface.}
      \label{fig:ex2}
    \end{figure}
}\end{example}

\begin{example}[cf.\cite{T}. $\gamma \in C^0$ and convex, but $\gamma \notin C^1$]\label{ex1}{\rm
Set
$\displaystyle \gamma(\nu)=\sum_{i=1}^{n+1}|\nu_i|$. 
Then,    
$\gamma \in C^0$ and convex, but $\gamma \notin C^1$. 
$W_\gamma$ is the cube
$
\{
x=(x_1, \cdots, x_{n+1}) \in {\bf R}^{n+1} \;|\;
{\max}\{|x_1|, \cdots, |x_{n+1}|\}=1
\}.
$

We show a picture (Figure \ref{fig:ex1}) for the special case where $n=1$. 
The Wulff shape $W_\gamma$ is the square drawn with straight line segments, and the  
Wulff plot is the dotted curves (Figure \ref{fig:ex1}, left). $\gamma$ is not differentiable at the four points $(\pm 1, 0), (0, \pm 1)$. The image of the Cahn-Hoffman map $\xi_\gamma:S^1\setminus \{(\pm 1, 0), (0, \pm 1)\} \to {\mathbb R}^2$ is the set of the four points $\{(1, 1), (-1, 1), (1, -1), (-1, -1)\}$ (Figure \ref{fig:ex1}, right). 

    \begin{figure}[H]
  \centering   \includegraphics[width=30mm,height=30mm,angle=0]{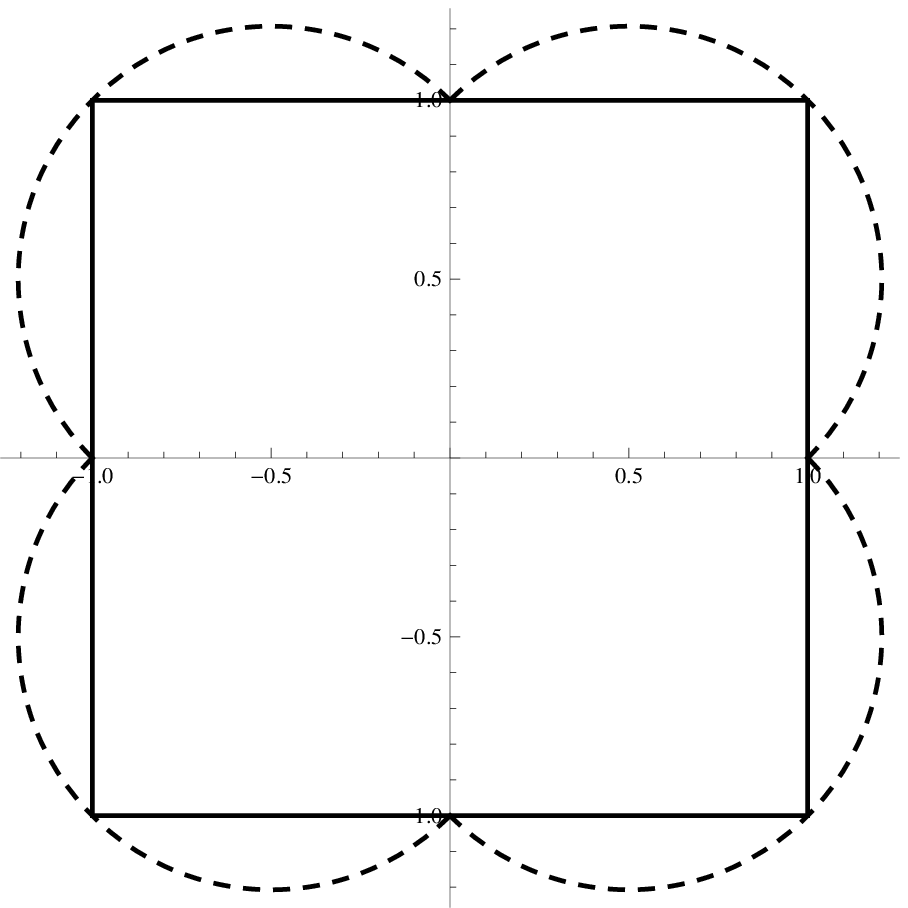}
\hspace{15mm}
\includegraphics[width=23mm,height=23mm,angle=0]{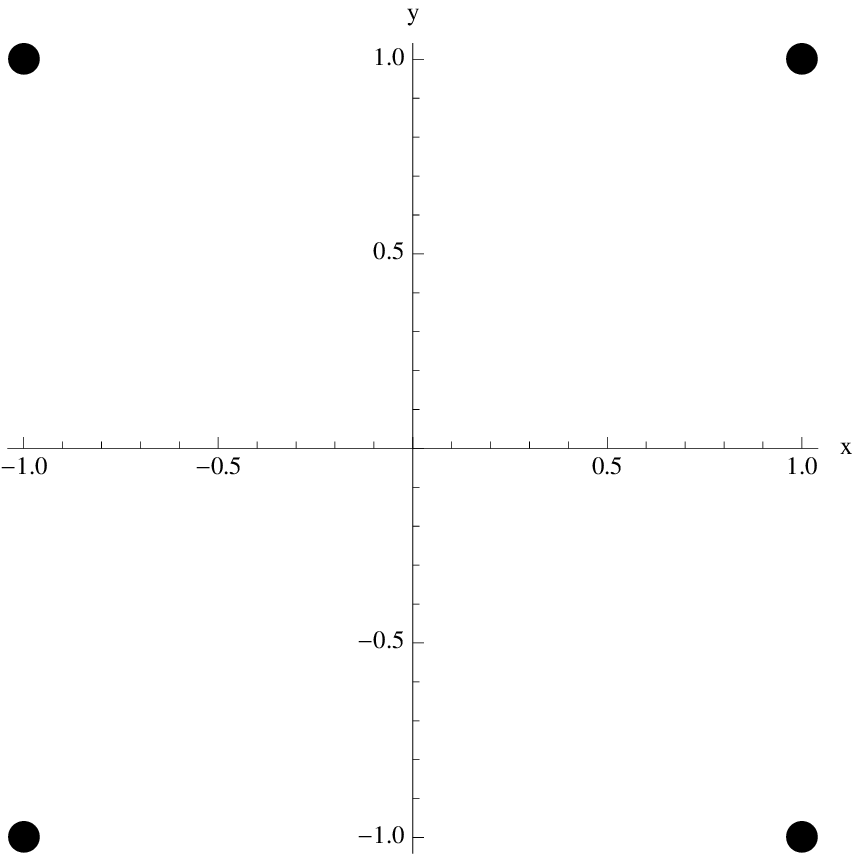}
\caption{Left: The Wulff shape (solid curves) and the Wulff plot (dotted curves) of Example \ref{ex1}. Right: The image of the Cahn-Hoffman map $\xi_{\gamma}:S^1\setminus \{(\pm 1, 0), (0, \pm 1)\} \to {\mathbb R}^2$.}
      \label{fig:ex1}
    \end{figure}
    
}\end{example}

\begin{example}[$\gamma \in C^0$ and convex, but $\gamma \notin C^1$]\label{ex-c}{\rm
Let $r$, $h$ be positive constants. Define $\gamma:S^n\to {\mathbb R}_{>0}$ as follows.
\begin{equation}\label{cyl}
\gamma(\nu)=
r\sqrt{
\nu_1^2+\cdots +\nu_n^2}+h|\nu_{n+1}|, \quad \nu=\nu_1, \cdots, \nu_{n+1} \in S^n.
\end{equation}
Then,    
$\gamma \in C^0(S^n)$ and convex, but $\gamma$ is not differentiable at any point in the following set $S(\gamma)$. 
\begin{equation}\label{sin}
S(\gamma):=\{
(\nu_1, \cdots, \nu_{n+1})\in S^n\;|\;\nu_{n+1}=0, \pm 1
\}.
\end{equation}
The Wulff shape $W_\gamma$ is the cylinder $C$ with radius $r$ and height $2h$ covered with two flat discs with radius $r$, that is
\begin{equation}\label{cy2} 
W_\gamma=
C\cup \Gamma_+\cup\Gamma_-,
\end{equation}
where
\begin{eqnarray}
C&=&
\{
(x_1, \cdots, x_{n+1}) \in {\mathbb R}^{n+1}\;|\;
\sqrt{x_1^2 + \cdots + x_n^2}=r, \ x_{n+1}\in [-h, h]
\},
\\
\Gamma_{\pm}&=&
\{
(x_1, \cdots, x_n, \pm h) \in {\mathbb R}^{n+1}\;|\;
\sqrt{x_1^2 + \cdots + x_n^2} \le r
\}.
\end{eqnarray} 
The Cahn-Hoffman map $\xi:S^n\setminus S(\gamma) \to {\mathbb R}^{n+1}$ is given as follows. 
\begin{equation}\label{cy3}
\xi(p\cos\theta, \sin\theta)=(rp, {\rm sgn}(\sin\theta)\cdot h), \quad p \in S^{n-1}, \ \theta \in (-\pi/2, 0)\cup (0, \pi/2),
\end{equation}
that is, $\xi$ maps a quater circle to one point, and the image $\xi(S^n\setminus S(\gamma))$ is the union of two circles with radius $r$ with height $\pm h$ as follows. 
\begin{equation}\label{xic}
\xi(S^n\setminus S(\gamma))=
\{
(x_1, \cdots, x_n, \pm h) \in {\mathbb R}^{n+1}\;|\;
\sqrt{x_1^2 + \cdots + x_n^2} = r
\}
\end{equation} 
(Figure \ref{fig:cyl}). 

Here we give a comment about the relationship between the singular points of $\gamma$ and the flat faces of $W_\gamma$. Each singular point $(0, \cdots, 0, \pm 1)$ corresponds to the $n$-dimensional disc $\Gamma_{\pm}$, and each singular point $(p, 0):=(\nu_1, \cdots, \nu_n, 0)$ corresponds to the straight line segment $\{(rp, x_{n+1}) \in {\mathbb R}^{n+1}\;|\;-h\le x_{n+1}\le h\}$.
\vskip0.8truecm

\begin{figure}[H]
  \centering   
  
\includegraphics[width=30mm,height=30mm,angle=0]{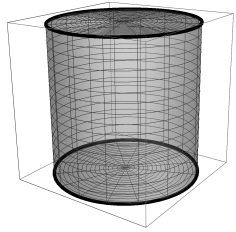}
  

      \caption{The Wulff shape (a cylinder covered with two flat discs) and the image of the Cahn-Hoffman map (two thick circles at the top and at the bottom) of Example \ref{ex-c}}
      \label{fig:cyl}
    \end{figure}
}\end{example}

\begin{example}[$\gamma \in C^\infty$ and convex, but not uniformly convex]\label{ex22}{\rm 
Set $n=1$. For $m \in {\mathbb N}$, 
define $\gamma_m(\nu):=(\nu_1^{2m}+\nu_2^{2m})^{1/(2m)}$. 
Then $\gamma$ is of $C^\infty$ and convex. 
$\gamma$ is uniformly convex for $m=1$, and it is not uniformly convex for $m\ge 2$. 

    \begin{figure}[H]
  \centering   \includegraphics[width=23mm,height=23mm,angle=0]{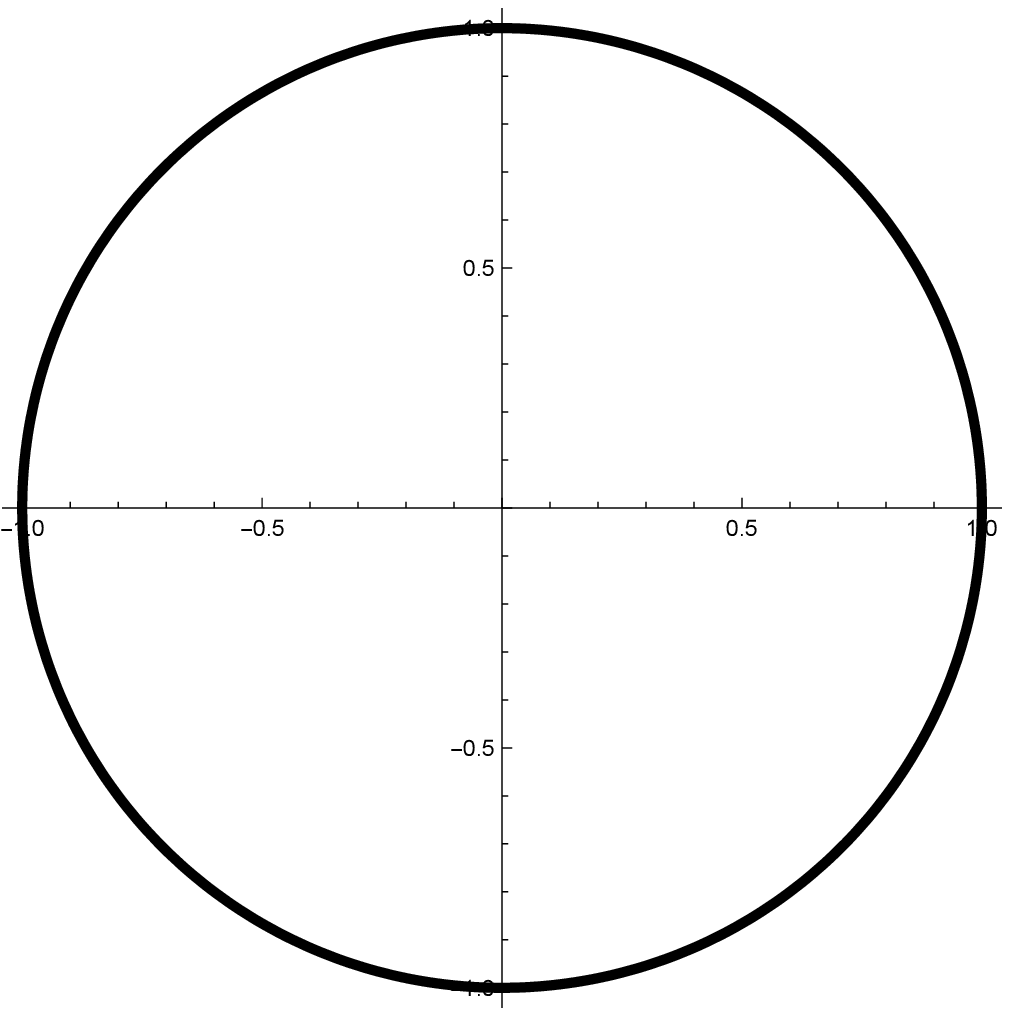}
\hspace{15mm}
\includegraphics[width=23mm,height=23mm,angle=0]{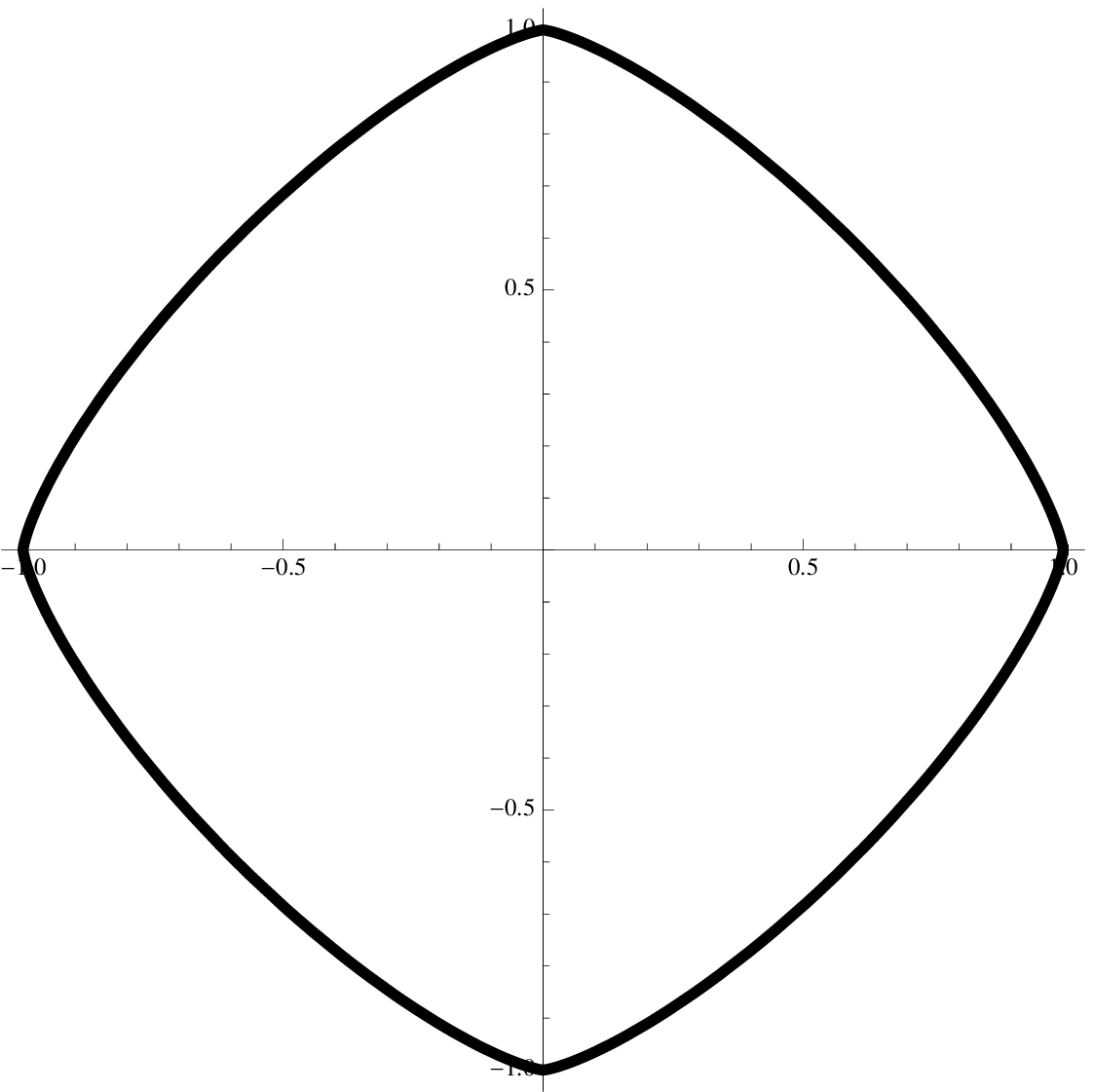}
\caption{Left: $W_\gamma$ for $m=1$, right: $W_\gamma$ for $m=2$, for $\gamma_m(\nu):=(\nu_1^{2m}+\nu_2^{2m})^{1/(2m)}$ in Example \ref{ex22}.}
      \label{fig:ex22}
    \end{figure}

By computations (see the appendix \S \ref{proof:ex22}), we obtain the following results. Set $\nu = (\cos\theta, \sin\theta)$. 
Then, 
the Cahn-Hoffman map $\xi_m$ for $\gamma_m$ is represented as follows.
\begin{eqnarray}
\xi_m(\cos\theta, \sin\theta)&:=& (f_m(\theta), g_m(\theta)) \nonumber\\
&=&
(\cos^{2m}\theta+\sin^{2m}\theta)^{(1/(2m))-1}(\cos^{2m-1}\theta, \sin^{2m-1}\theta). \label{kusi2}
\end{eqnarray}
Also we have
\begin{eqnarray}
A_m&:=&
D^2\gamma_m +\gamma_m \cdot 1
\nonumber \\
&=&(2m-1)\cos^{2m-2}\theta \sin^{2m-2}\theta(\cos^{2m}\theta+\sin^{2m}\theta)^{(1/(2m))-2}.
\end{eqnarray}
Hence, 

(i) If $m \ge 2$, $A_m=0$ at $\theta=(1/2)\ell \pi$, ($\ell\in \mathbb Z$).

(ii) $A_m$ is positive-definite on $S^1\setminus\{(\cos\theta, \sin\theta) \:|\:
\theta=(1/2)\ell \pi, (\ell\in \mathbb Z)\}$. 
\newline The curvature $\kappa_m$ of $\xi_m$ with respect to the outward-pointing normal $\nu$ is represented as
\begin{eqnarray}
\kappa_m &=&
\frac{-f_m'g_m''+f_m''g_m'}{((f_m')^2+(g_m')^2)^{3/2}} \nonumber \\
&=&\frac{-1}{2m-1}\cos^{-2m+2}\theta\sin^{-2m+2}\theta(\cos^{2m}\theta+\sin^{2m}\theta)^{2-\frac{1}{2m}}.
\end{eqnarray}
Hence, for any $\ell \in {\mathbb Z}$ and $m \ge 2$, 
\begin{equation}
\lim_{\theta \to \frac{\ell}{2}\pi}\kappa_m(\theta)=-\infty
\end{equation}
holds. 
However, since
\begin{equation}
\kappa_m \;ds
=-d\theta
\end{equation}
holds,  
\begin{equation}\label{tot10}
\int_0^{2\pi}\kappa_m\;ds
=-\int_0^{2\pi}d\theta
=-2\pi.
\end{equation}
However, (\ref{tot10}) is trivial because $\xi_m:S^1\to{\mathbb R}^2$ is a front with $\nu \in S^1$.
}\end{example}
    
\begin{example}[$\gamma \in C^1$ and convex, but $\gamma \notin C^2$]\label{ex5}{\rm 
Set $n=1$. 
Let $C_r$ be the convex closed curve passing the four points 
$(1, 1), (1, -1), (-1, 1), (-1, -1)$ which is the union of four circular arcs with radius $r\ge\sqrt{2}$ (Figure \ref{fig:ex5}, left). $C_{\sqrt{2}}$ is the circle with center at $(0, 0)$. 
For simplicity, we regard a straight line segment as a circular arc with radius $r=\infty$. We will write any point $\nu \in S^1$ as $\nu = (\cos\theta, \sin\theta)$. 
The support function $\gamma_r:S^1 \to {\mathbb R}_{>0}$ of $C_r$ is given by the symmetric extension of the following equation.
\begin{eqnarray}
\gamma_r(\theta)=\left\{ \begin{array}{ll}
\cos\theta+\sin\theta, & \pi/4 \le \theta \le (\pi/2) - \sin^{-1}(1/r),\\
r+(1-\sqrt{r^2-1})\sin\theta, & (\pi/2) - \sin^{-1}(1/r) < (\pi/2) + \sin^{-1}(1/r), \\
|\cos\theta|+\sin\theta, & (\pi/2) + \sin^{-1}(1/r) \le \theta \le (3/4)\pi.\\
\end{array} \right.
\end{eqnarray} 
Note
$$
\gamma_\infty(\theta):=\lim_{r \to \infty} \gamma_r(\theta)=|\cos\theta|+|\sin\theta|, \quad 0 \le \theta \le 2\pi,
$$
which coincides with $\gamma$ in Example \ref{ex1}. 

When $\sqrt{2} < r <\infty$, by computation we can show that $\gamma_r \in C^1$ and convex, but $\gamma_r \notin C^2$. 

If $\sqrt{2} \le r <\infty$, both $\xi_{\gamma_r}(S^1)$ and $W_{\gamma_r}$ coincide with $C_r$. 
When, $r=\infty$, The image of the Cahn-Hoffman map $\xi_{\gamma_{\infty}}:S^1\setminus \{(\pm 1, 0), (0, \pm 1)\} \to {\mathbb R}^2$ is the set of the four points $\{(1, 1), (-1, 1), (1, -1), (-1, -1)\}$ (Figure \ref{fig:ex5}, right).
    
    \begin{figure}[H]
  \centering   \includegraphics[width=30mm,height=30mm,angle=0]{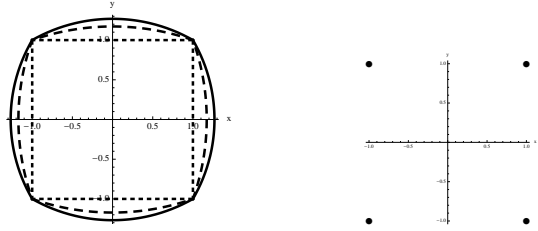}
\hspace{15mm}
\includegraphics[width=23mm,height=23mm,angle=0]{fourPoints3RC.eps}
\caption{Left: The closed curves $C_r$ in Example \ref{ex5}. 
The solid curve: $r=2$，the dashed curve: $r=3$, the dotted curve: $r=+\infty$. 
When $\sqrt{2} \le r <\infty$, $C_r=W_{\gamma_r}=\xi_{\gamma_r}(S^1)$. Right: The image of the Cahn-Hoffman map $\xi_{\gamma_{\infty}}:S^1\setminus \{(\pm 1, 0), (0, \pm 1)\} \to {\mathbb R}^2$.}
      \label{fig:ex5}
    \end{figure}
}\end{example}
    
\begin{example}[$\gamma \in C^\infty$, and $\gamma$ is not convex]\label{ex4'}{\rm 
Set $n=1$. Define $\gamma$ as 
$\gamma(\nu)=4\nu_1^3-3\nu_1+2$. 
Then, 
$\gamma \in C^\infty$, and $\gamma$ is not convex. 
The whole of the closed curve with self-intersection in Figure \ref{fig:ex4'} is the image $\hat{W}_\gamma:=\xi_\gamma(S^1)$ of the Cahn-Hoffman map $\xi_\gamma$, while the closed convex solid curve that is a proper subset of $\hat{W}_\gamma$ is the Wulff shape $W_\gamma$. By computation, we obtain
\begin{equation}
\gamma(\cos\theta, \sin\theta)=4\cos^3\theta-3\cos\theta+2=:\gamma(\theta),
\end{equation}
\begin{equation}
\xi_\gamma(\theta)=(8\cos^2\theta\sin^2\theta+4\cos^2\theta+2\cos\theta-3, 
-8\cos^3\theta\sin\theta+2\sin\theta)=:(f(\theta), g(\theta)),
\end{equation}
\begin{equation}
A:=d\xi_\gamma
=2(-16\cos^3\theta+12\cos\theta+1)(-\sin\theta, \cos\theta)
=:a(\theta)(-\sin\theta, \cos\theta),
\end{equation}
\begin{equation}
\kappa_\gamma(\theta):=
\frac{-f'g''+f''g'}{((f')^2+(g')^2)^{3/2}}
=\frac{-1}
{2|-16\cos^3\theta+12\cos\theta+1)|}
=\frac{-1}
{|a(\theta)|}.
\end{equation}
Hence, for any $\rho \in a^{-1}(0)$, 
\begin{equation}
\lim_{\theta(\notin a^{-1}(0))\to\rho}\kappa_\gamma(\theta)=-\infty.
\end{equation}
$a^{-1}(0) \subset S^1$ is the set of exactly six points, and they correspond to the singular points of $\xi_\gamma$ (Figure \ref{fig:ex4'}). 

\begin{figure}[H]
  \centering
   \includegraphics[width=45mm,height=45mm,angle=0]{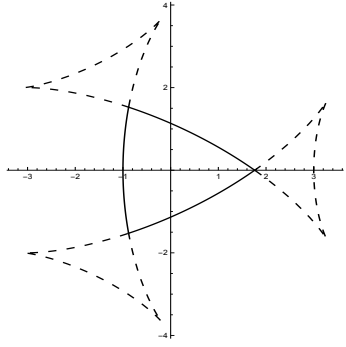}
      \caption{The image of the Cahn-Hoffman map $\xi_\gamma$ for $\gamma$ in Example \ref{ex4'}. The six vertices are the image of the singular points of $\xi_\gamma$. The closed convex solid curve is the Wulff shape $W_\gamma$.}
      \label{fig:ex4'}
    \end{figure}
}\end{example}

\section{Anisotropic shape operator and anisotropic curvatures}\label{a-curv}

Assume that $\gamma:S^n \to {\mathbb R}_{>0}$ is of $C^2$. 
Let $M_0$ be an oriented compact connected $n$-dimensional $C^\infty$ manifold with smooth boundary $\partial M_0$. 
As in \S \ref{1st}, we assume that  a map $X:M_0\to {\mathbb R}^{n+1}$ satisfies the conditions (A1), (A2), and (A3) in \S \ref{pre} with $r=2$, $X_i=X$, $M_i=M_0$, and $\nu_i=\nu$.

\begin{definition}\label{CHF}{\rm
The Cahn-Hoffman field along $X$ (for $\gamma$) is defined as
\begin{equation}\label{CHFX}
\tilde{\xi}:=\tilde{\xi}_\gamma:=\xi\circ \nu:M_0 \to {\mathbb R}^{n+1},
\end{equation}
where
$
\xi(\nu):=\xi_\gamma(\nu) =D\gamma+\gamma(\nu)\nu
$, ($\nu \in S^n$), 
is the Cahn-Hoffman map for $\gamma$. 
$\tilde{\xi}_\gamma$ is called also the anisotropic Gauss map of $X$ (for $\gamma$).
}
\end{definition}

\begin{definition}[anisotropic shape operator, cf. \cite{R}, \cite{HLMG2009}]\label{shape}{\rm 
(i) The linear map $S^\gamma_p:T_p M_0 \to T_p M_0$ given by the $n\times n$ matrix 
$
S^\gamma:=-d\tilde{\xi}_\gamma
$
is called the anisotropic shape operatior of $X$.

(ii) If $S^\gamma_p$ is proportional to the identity map, we say that $p$ is anisotropic-umbilic. 
}\end{definition}

%

\begin{definition}[anisotropic curvatures, cf. \cite{R}, \cite{HLMG2009}]\label{curv}{\rm 
(i) The eigenvalues of $S^\gamma$ are called the anisotropic principal
curvatures of $X$. We denote them by $k^\gamma_1, \cdots, k^\gamma_n$.

(ii) Let $\sigma^\gamma_r$ be the elementary symmetric functions of $k^\gamma_1, \cdots, k^\gamma_n$, that is
\begin{equation}\label{curv2}
\sigma^\gamma_r :=\sum_{1\le l_1< \cdots <l_r\le n} k^\gamma_{l_1}\cdots k^\gamma_{l_r}, \quad r=1, \cdots, n.
\end{equation}
We set $\sigma^\gamma_0:=1$. 
The $r$-th anisotropic mean curvature $H^\gamma_r$ of $X$ is defined by $H^\gamma_r:=\sigma^\gamma_r/{}_nC_r$.

(iii) $\Lambda^\gamma:=H^\gamma_1$ is called  the anisotropic mean curvature of $X$. $\Lambda^\gamma$ is often written as $\Lambda$ for simplicity.  
}\end{definition}

\begin{remark}\label{rcurv}{\rm 
Anisotropic principal curvatures of $\xi_\gamma$ are $-1$. 
Hence, each $r$-th anisotropic mean curvature of $\xi_\gamma$ is $(-1)^r$.
}\end{remark}  

\begin{lemma}\label{HL}
(i) (\cite{HL2008}) If $A:=d\xi=D^2 \gamma+\gamma \cdot 1$ is positive definite, then all of the anisotropic principal curvatures of $X:M_0\to {\mathbb R}^{n+1}$ are real. 

(ii) $k^\gamma_i$ is not a real value in general. However, 
each $H^\gamma_r$ is always a real valued function on $M_0$. 
\end{lemma}

\noindent {\it Proof. } \ 
The outline of the proof of (i) was given in \cite[p.699]{HL2008}. We will  give a detailed proof in \S \ref{pfd}. 
(ii) is proved as follows. 
Denote by $I_n$ the $n\times n$ identity matrix. 
We compute
\begin{eqnarray}
\det(I_n+\tau S^\gamma)&=&
(1+\tau k^\gamma_1)\cdots(1+\tau k^\gamma_n) \\
&=&\sum_{r=0}^n \tau^r \sum_{1 \le l_1< \cdots l_r\le n}
k^\gamma_{l_1}\cdots k^\gamma_{l_r}\\
&=& \sum_{r=0}^n \tau^r \sigma^\gamma_r 
= \sum_{r=0}^n ({}_nC_r) H^\gamma_r\tau^r, \label{exp3}
\end{eqnarray}
which shows that each $H^\gamma_r$ is a real valued function on $M_0$. 
\hfill $\Box$

\begin{remark}[cf. \cite{KP2005}]\label{AMC2}{\rm 
Take any point $p \in M_0^o$. Assume that $\nu(p)$ is a regular point of $(D^2\gamma+\gamma\cdot 1)$. We compute the anisotropic mean curvature $\Lambda(p)$ of $X$ at $p$. 
Let $\{e_1, \cdots, e_n\}$ be a locally defined
frame on $S^{n}$ such that $(D^2\gamma+\gamma\cdot 1)e_i=(1/\mu_i)e_i$, where $\mu_i$ are the principal curvatures of $\xi_\gamma$ with respect to $\nu$. 
Note that the basis
$\{e_1, \cdots, e_n\}$ at $\nu(p)$ also serves as an orthogonal basis for the tangent hyperplane of $X$ at $p$. Let $(-w_{ij})$ be the matrix
representing $d\nu$ with respect to this basis. Then
$$
-S^\gamma=d\xi_\gamma\circ d\nu=(D^2\gamma+\gamma\cdot  1)d\nu= \left(\begin{array}{lllll}
-w_{11}/\mu_1 \quad \cdots \quad   -w_{1n}/\mu_1 \\
\ \qquad \cdot \ \ \ \quad \quad \cdots \ \ \quad\qquad \cdot\\
\ \qquad \cdot \ \ \ \quad \quad \cdots \ \ \quad\qquad \cdot\\
\ \qquad \cdot \ \ \ \quad \quad \cdots \ \ \quad\qquad \cdot\\
-w_{n1}/\mu_n \quad \cdots \quad   -w_{nn}/\mu_n
\end{array} \right). $$
This with (\ref{trace}) gives 
\begin{equation}\label{eqR2}
\Lambda=(1/n)(w_{11}/\mu_1 +\cdots + w_{nn}/\mu_n).
\end{equation}
}\end{remark}

For later use, we give some notations and formulas using local coordinates $(u_1, \cdots, u_n)$ in $M_0$.
Set as usual
$$
g_{ij}:=\langle X_i, X_j\rangle, \ 
h_{ij}:=-\langle \nu_i, X_j\rangle
=-\langle X_i, \nu_j\rangle, 
$$
where
$$
X_i=X_{u_i}, \ {\rm etc.}
$$
At any regular point of $X$, we set
$$
(g^{ij})=(g_{ij})^{-1}. 
$$
By using the Einstein summation convention, we have
$$
\nu_i=-h_{ij}g^{jk}X_k,  
$$
and $d\nu$ can be represented as an $n\times n$ matrix
$$
d\nu=-(h_{ij})(g^{ij}).
$$

For the Cahn-Hoffman map $
\xi:S^n \to {\mathbb R}^{n+1}$, 
consider its differential 
$$
(d\xi)_{\nu}:T_{\nu} S^n \to T_{\xi(\nu)}{\mathbb R}^{n+1}={\mathbb R}^{n+1}, \quad 
d\xi=D^2 \gamma+\gamma \cdot 1.
$$
For the Cahn-Hoffman field $\tilde{\xi}:=\xi\circ \nu:M \to {\mathbb R}^{n+1}$ along $X$, set
\begin{equation}\label{dxi2}
-\tilde{h}_{ij}:=\langle \tilde{\xi}_i, X_j\rangle, \quad \tilde{\xi}_i:=\tilde{\xi}_{u_i}.
\end{equation}
Then, it is easy to show that
\begin{equation}\label{dxi3}
\tilde{\xi}_i=-\tilde{h}_{il}g^{lj}X_j
\end{equation}
holds. 
In fact, since $\tilde{\xi}_i$ is perpendicular to $\nu$ (Proposition \ref{PCH}), we can write
\begin{equation}\label{dxi4}
\tilde{\xi}_i=a_i^j X_j,
\end{equation}
which gives
$$
-\tilde{h}_{il}=\langle \tilde{\xi}_i, X_l\rangle=\langle a_i^j X_j, X_l\rangle =a_i^jg_{jl},
$$
and so we have
$$
(-\tilde{h}_{il})=(a_i^j)(g_{jl}). 
$$
Hence
$$
(a_i^j)=(-\tilde{h}_{il})(g^{lj}),
$$
which gives (\ref{dxi3}). 
Therefore $d\tilde{\xi}$ is represented as an $n\times n$ matrix
\begin{equation}\label{dxi5}
d\tilde{\xi}=-(\tilde{h}_{ij})(g^{ij}).
\end{equation}
Since $S^\gamma=
-d\tilde{\xi}$, we have the following. 
\begin{lemma}\label{shapeL}
$S^\gamma=
(\tilde{h}_{ij})(g^{ij})$ holds.
\end{lemma}

\section{First variation formula, anisotropic mean curvature, and anisotropic Gauss map}\label{1st}

In this section we assume that $\gamma:S^n \to {\mathbb R}_{>0}$ is of $C^2$. 

Let $M_0$ be an oriented compact connected $n$-dimensional $C^\infty$ manifold with smooth boundary $\partial M_0$. 
We assume that a map $X:M_0\to {\mathbb R}^{n+1}$ satisfies (A1), (A2), and (A3) in \S \ref{pre} with $r=2$, $X_i=X$, $M_i=M_0$, and $\nu_i=\nu$. 

Recall that the Cahn-Hoffman field along $X$ (for $\gamma$) is defined as
\begin{equation}\label{CHFX2}
\tilde{\xi}:=\tilde{\xi}_\gamma:=\xi_\gamma\circ \nu:M_0 \to {\mathbb R}^{n+1},
\end{equation}
where
$
\xi_\gamma(\nu) =D\gamma+\gamma(\nu)\nu
$, ($\nu \in S^n$), 
is the Cahn-Hoffman map for $\gamma$. 

\begin{remark}\label{AMC}{\rm
The anisotropic mean curvature $\Lambda$ of $X$ at any regular point of $X$ can be represented as
$$\Lambda =\frac{1}{n}(-{\rm div}_{M_0} D\gamma+nH\gamma),
$$
where $H$ is the mean curvature of $X$ (cf. \cite{R}, \cite{KP2005}). 
}\end{remark}

First we give the first variation formula of the anisotropic surface energy for immersions.

\begin{lemma}\label{firstFF}
Assume that $X:M_0\to {\mathbb R}^{n+1}$ is an immersion 
and let
$X_\epsilon:M_0\to {\mathbb R}^{n+1}$ be a variation of $X=X_0$. 
We assume for simplicity that $X_\epsilon$ is of $C^\infty$  in $\epsilon$. Set
$$
\delta X := \frac{\partial X_\epsilon}{\partial \epsilon}\Big|_{\epsilon=0}, \quad \psi:=\big\langle \delta X, \nu\big\rangle.
$$ 
Then the first variation of the anisotropic energy $\mathcal F_\gamma$ is given as follows. 
\begin{eqnarray} \delta {\mathcal F_\gamma} &:=&
\frac{d{\mathcal F}_\gamma(X_\epsilon)}{d\epsilon}\Big|_{\epsilon=0} \nonumber\\
&=&- \int_{M_0} n\Lambda\psi\;dA
-\oint_{\partial {M_0}} 
\langle \tilde{\xi}, -\langle \delta X, \nu\rangle N + \langle \delta X, N \rangle\nu\rangle \;d{\tilde s}, \label{firstF1}\end{eqnarray}
where 
$dA$ 
is the $n$-dimensional volume form of $M_0$ induced by $X$, $N$ is the outward-pointing unit conormal along $\partial M_0$, and  
$d{\tilde s}$ is the $(n-1)$-dimensional volume form of $\partial M_0$ induced by $X$. 
Denote by $R$ the $\pi/2$-rotation on the $(N, \nu)$-plane. Denote by $p$ the projection from ${\mathbb R}^{n+1}$ to the $(N, \nu)$-plane. 
Then 
\begin{equation}\label{firstF2}
\delta {\mathcal F_\gamma}
=
- \int_{M_0} n\Lambda\psi\;dA
-\oint_{\partial {M_0}} \langle \delta X, R(p(\tilde{\xi}))\rangle\;d{\tilde s}.
\end{equation}
\end{lemma}

\noindent {\it Proof.} \ 
We can write
$$
X_{\epsilon} =X+\epsilon(\eta+\psi\nu)+{\mathcal O}(\epsilon^2),
$$
where $\eta$ is the tangential component and $\psi\nu$ is the normal component of the variation vector field $\delta X$ of $X_\epsilon$. 
Then the first variation of the anisotropic energy $\mathcal F_\gamma$ is computed as follows. 
\begin{eqnarray} \delta {\mathcal F_\gamma} &:=&
\frac{d{\mathcal F}_\gamma(X_\epsilon)}{d\epsilon}\Big|_{\epsilon=0} \nonumber\\
&=&\int_{M_0} \langle D\gamma, -\nabla\psi+d\nu(\eta)\rangle +\gamma(-nH\psi+{\rm div}\eta)\; dA \label{imp1}\\
&=&\int_{M_0} \langle D\gamma, -\nabla\psi\rangle -nH\gamma\psi\; dA
+
\int_{M_0} \langle D\gamma, d\nu(\eta)\rangle +\gamma{\rm div}\eta\; dA \label{imp2}\\
&=& \int_{M_0} \psi({\rm div}_{M_0} D\gamma-nH\gamma )\;dA
+\oint_{\partial {M_0}} -\psi\langle D\gamma,N\rangle +\gamma\langle \eta,N \rangle \;d{\tilde s}. \label{firstF}
\end{eqnarray}

We compute the integrand of the boundary integral in (\ref{firstF}):
\begin{eqnarray}
f&:=& -\psi\langle D\gamma,N\rangle +\gamma\langle \eta,N \rangle \label{ft01}\\
&=& -\langle \delta X, \nu\rangle \langle \tilde{\xi}, N\rangle +\langle \tilde{\xi}, \nu\rangle \langle \delta X, N \rangle \label{ft02}\\
&=&
\langle \tilde{\xi}, -\langle \delta X, \nu\rangle N + \langle \delta X, N \rangle\nu\rangle. \label{ft03}
\end{eqnarray} 
(\ref{ft03}) with (\ref{firstF}) gives (\ref{firstF1}). 

From (\ref{ft02}), we obtain also the following:
\begin{equation}
f= \langle
\delta X, 
-\langle \tilde{\xi}, N\rangle \nu + \langle \tilde{\xi}, \nu\rangle N
\rangle.
\end{equation}

By using 
\begin{equation}
R(N)=\nu, \quad R(\nu)=-N,
\end{equation}
we compute
\begin{eqnarray}
f&=&
\langle
\delta X, 
-\langle \tilde{\xi}, N\rangle \nu + \langle \tilde{\xi}, \nu\rangle N
\rangle\\
&=&
\langle
\delta X, 
-\langle \tilde{\xi}, N\rangle R(N) - \langle \tilde{\xi}, \nu\rangle R(\nu)
\rangle\\
&=&
\langle
\delta X, 
-R(\langle \tilde{\xi}, N\rangle N + \langle \tilde{\xi}, \nu\rangle \nu)
\rangle\\
&=&
\langle
\delta X, 
-R(p(\tilde{\xi}))
\rangle. \label{int2}
\end{eqnarray} 
(\ref{firstF}) with (\ref{int2}) gives
\begin{eqnarray}\label{fst}
\delta {\mathcal F_\gamma}
&=&
\int_{M_0} \psi({\rm div}_{M_0} D\gamma-nH\gamma )\;dA
-\oint_{\partial {M_0}} \langle \delta X, R(p(\tilde\xi))\rangle\;d{\tilde s}\\
&=&
- \int_{M_0} n\Lambda\psi\;dA
-\oint_{\partial {M_0}} \langle \delta X, R(p(\tilde\xi))\rangle\;d{\tilde s} \label{ft},
\end{eqnarray}
which gives (\ref{firstF2}).
\hfill $\Box$

\vskip0.5truecm

Next we give the first variation formula of the anisotropic surface energy for piecewise smooth weak immersions, which is a generalization of Lemma \ref{firstFF}.

\begin{proposition}\label{FV} 
We assume that the map $X:M_0\to {\mathbb R}^{n+1}$ satisfies (A1), (A2), and (A3) in \S \ref{pre} with $r=2$, $X_i=X$, $M_i=M_0$, and $\nu_i=\nu$. 
Let 
$X_\epsilon:M_0\to {\mathbb R}^{n+1}$, ($\epsilon \in J:=[-\epsilon_0, \epsilon_0]$),  be a variation of $X$, that is, $\epsilon_0>0$ and $X_0=X$. 
We assume for simplicity that $X_\epsilon$ is of $C^\infty$  in $\epsilon$. Then, each $X_\epsilon$ satisfies (A2) and (A3) in \S \ref{pre} with $r=2$, $X_i=X_\epsilon$, $M_i=M_0$, and $\nu_i=\nu_\epsilon$. 
We also assume that, for each $\epsilon \in J$, the anisotropic mean curvature $\Lambda_\epsilon$ of $X_\epsilon$ is bounded on $M_0^o$. Set
$$
\delta X := \frac{\partial X_\epsilon}{\partial \epsilon}\Big|_{\epsilon=0}, \quad \psi:=\big\langle \delta X, \nu\big\rangle.
$$ 
Then the first variation of the anisotropic energy $\mathcal F_\gamma$ is given as follows. 
\begin{equation} \label{firstF22}
\delta {\mathcal F_\gamma} =
\frac{d{\mathcal F}_\gamma(X_\epsilon)}{d\epsilon}\Big|_{\epsilon=0}
=
- \int_{M_0} n\Lambda\psi\;dA
-\oint_{\partial {M_0}} \langle \delta X, R(p(\tilde{\xi}))\rangle\;d{\tilde s},
\end{equation}
where 
$dA$, $d{\tilde s}$, $R$, $p$ are the same as those in Lemma \ref{firstFF}.
\end{proposition}

The proof of Proposition \ref{FV} will be given in \S \ref{PFV}.

It is known that the first variation formula of the $(n+1)$-dimensional volume
enclosed by $X_\epsilon$ is given as follows (cf. \cite{BDE1988}).
\begin{equation}\label{firstv}
\delta V=\int_{M_0} \psi\;dA.
\end{equation}

Now let $X:M=\cup_{i=1}^k M_i \to{\mathbb R}^{n+1}$ be a piecewise-$C^2$ weak immersion (for definition, see \S \ref{pre}) with unit normal $\nu_i:M_i\to S^n$. 
We will use the following notations. 
$$
X_i:=X|_{M_i},
$$
$$
\tilde{\xi}_i:=\xi\circ\nu_i:M_i \to {\mathbb R}^{n+1}.
$$
Denote by $N_i$ the outward-pointing unit conormal along $\partial M_i$. 
Denote by $R_i$ the $\pi/2$-rotation on the $(N_i, \nu_i)$-plane. Denote by $p_i$ the projection from ${\mathbb R}^{n+1}$ to the $(N_i, \nu_i)$-plane. 

(\ref{firstF2}) together with (\ref{firstF22}), (\ref{firstv}) gives the Euler-Lagrange equations in the following Proposition \ref{EL}. 

\begin{proposition}[Euler-Lagrange equations. For $n=2$, see B. Palmer \cite{P2012}]\label{EL} 
A piecewise-$C^2$ weak immersion $\displaystyle X:M=\cup_{i=1}^k M_i \rightarrow{\mathbb R}^{n+1}$ 
is a critical point of the anisotropic energy $\displaystyle {\mathcal F}_\gamma(X)=\int_M \gamma(\nu) \:dA$ for $(n+1)$-dimensional volume-preserving variations if and only if 

(i) \ The anisotropic mean curvature $\Lambda$ of $X$ is constant on $M\setminus S(X)$, and 

(ii) $\tilde{\xi}_i(\zeta)-\tilde{\xi}_j(\zeta) \in T_\zeta M_i\cap T_\zeta M_j =T_\zeta(\partial M_i\cap \partial M_j)$ holds at any  $\zeta \in \partial M_i\cap \partial M_j$, here a tangent space of a submanifold of ${\mathbb R}^{n+1}$ is naturally identified with a linear subspace of ${\mathbb R}^{n+1}$. 
\end{proposition}

\noindent {\it Proof.} \ 
The first term of the right-hand side of (\ref{firstF2}) together with (\ref{firstv}) clearly gives (i). 
By the second term of the right-hand side of (\ref{firstF22}), 
the condition on $\partial M_i\cap \partial M_j$ is
\begin{equation}\label{bdryc}
R_i(p_i(\tilde{\xi}_i))+R_j(p_j(\tilde{\xi}_j))= 0.
\end{equation}
Since, at any point $\zeta \in\partial M_i\cap \partial M_j$, 
$(N_i, \nu_i)$-plane coincides with $(N_j, \nu_j)$-plane, 
(\ref{bdryc}) is equivalent to 
\begin{equation}\label{proj}
p_i(\tilde{\xi}_i)-p_j(\tilde{\xi}_j)= 0.
\end{equation}
Again since, at any point $\zeta \in\partial M_i\cap \partial M_j$, 
$(N_i, \nu_i)$-plane coincides with $(N_j, \nu_j)$-plane, (\ref{proj}) implies (ii). 
\hfill $\Box$

\begin{cor}[Euler-Lagrange equations for the isotropic case]\label{ELI} 
A piecewise-$C^2$ weak immersion $\displaystyle X:M=\cup_{i=1}^k M_i \rightarrow{\mathbb R}^{n+1}$ is a critical point of the volume (that is the $n$-dimensional area) $\displaystyle {\mathcal A}(X)=\int_M \:dA$ for $(n+1)$-dimensional volume-preserving variations if and only if 

(i) \ The mean curvature $H$ of $X$ is constant on $M$, and 

(ii) $\displaystyle X:M\rightarrow{\mathbb R}^{n+1}$ is a $C^\omega$ immersion. 
\end{cor}

\noindent {\it Proof.} \ 
(i) is an immediate consequence of Proposition \ref{EL} (i). 

We prove (ii). Proposition \ref{EL} (ii) implies that 
\begin{equation}\label{bdc}
\nu_i(\zeta)-\nu_j(\zeta) \in T_\zeta M_i\cap T_\zeta M_j =T_\zeta(\partial M_i\cap \partial M_j)
\end{equation} 
holds at any  $\zeta \in \partial M_i\cap \partial M_j$. 
However, both of $\nu_i(\zeta)$ and $\nu_j(\zeta)$ are orthogonal to $T_\zeta(\partial M_i\cap \partial M_j)$, which implies that $\nu_i(\zeta)-\nu_j(\zeta)=0$. 
Therefore, $\nu$ is continuous on $M$. Because of the ellipticity of the equation $H=$ constant, $\displaystyle X:M\rightarrow{\mathbb R}^{n+1}$ is a $C^\omega$ immersion. 
\hfill $\Box$

\begin{definition}\label{defR}{\rm
A piecewise-$C^2$ weak immersion $\displaystyle X:M=\cup_{i=1}^k M_i \rightarrow{\mathbb R}^{n+1}$ is called a hypersurface with constant anisotropic mean curvature (CAMC for short) if $X$ satisfies the conditions (i) and (ii) in Proposition \ref{EL}. 
}
\end{definition}

If $\gamma$ is convex, then we have the following striking result that the Cahn-Hoffman field along any piecewise-$C^2$ CAMC hypersurface is defined continuously over the whole hypersurface as follows.  

\begin{theorem}\label{CHP}
Assume $\gamma:S^n\to {\mathbb R}_{>0}$ is of $C^2$ and the convex integrand of its Wulff shape $W_\gamma$. 
Let $X:M=\cup_{i=1}^k M_i\to {\mathbb R}^{n+1}$ be a closed piecewise-$C^2$ CAMC hypersurface with unit normal $\nu_i:M_i \to S^n$, ($i=1, \cdots k$).  
Then, the collection of the Cahn-Hoffman fields $\tilde{\xi}_i:=\xi\circ \nu_i :M_i \to {\mathbb R}^{n+1}$, ($i=1, \cdots k$) defines a $C^0$ map $\tilde{\xi}:M\to {\mathbb R}^{n+1}$. 
\end{theorem}

\noindent {\it Proof}. \ 
Since $\gamma$ is convex and of $C^2$, $W_\gamma=\hat{W}_\gamma$ and it is strictly convex (\cite{HN2017}), that is, for any two different points $q_1, q_2 \in W_\gamma$, the straight line containing $q_1, q_2$ intersects $W_\gamma$ only at $q_1, q_2$.

We consider the boundary condition in Proposition \ref{EL} (ii), which is 
\begin{equation}\label{bdry}
\tilde{\xi}_i(q)-\tilde{\xi}_j(q) \in T_q(\partial M_i\cap \partial M_j), \quad \forall q \in \partial M_i\cap \partial M_j.
\end{equation} 
Take an arbitrary point $q\in \partial\Sigma_i\cap \partial\Sigma_j$. 
Assume that $\tilde{\xi}_i(q)\ne \tilde{\xi}_j(q)$. 
Then, from (\ref{bdry}), we have
$$
\langle \tilde{\xi}_i(q)-\tilde{\xi}_j(q), \nu_i(q)\rangle=0
=
\langle \tilde{\xi}_i(q)-\tilde{\xi}_j(q), \nu_j(q)\rangle.
$$
Hence,  
\begin{equation}\label{bdry2}
(\tilde{\xi}_i(q)-\tilde{\xi}_j(q)) \in T_{\tilde{\xi}_i(q)}W, \quad (\tilde{\xi}_i(q)-\tilde{\xi}_j(q))  \in T_{\tilde{\xi}_j(q)}W
\end{equation}
holds. 
However, (\ref{bdry2}) contradicts the fact that $W$ is strictly convex. 
Hence $\tilde{\xi}_i(q)= \tilde{\xi}_j(q)$ holds, and we have proved the desired result. 
\hfill$\Box$


\begin{remark}\label{AGM}{\rm 
It is shown that, at any regular point of $X$, it holds that 
\begin{equation}\label{trace}
\Lambda=-\frac{1}{n}{\rm trace}_M (D^2\gamma+\gamma\cdot 1)\circ d\nu
= -\frac{1}{n}{\rm trace}_M\;d(\xi_\gamma\circ\nu)
= -\frac{1}{n}{\rm trace}_M\;d(\tilde\xi_\gamma)
\end{equation}
(cf. \cite{KP2005}). 
Hence, at points where $\gamma$ is strictly convex, by (\ref{trace}),  the equation ``$\Lambda = \textup{constant}$'' is elliptic. 
}\end{remark}

\begin{remark}\label{homoge3}{\rm 
Here we give a known representation formula for the anisotropic mean curvature for a graph for readers' convenience. We denoted the homogeneous extension of $\gamma : S^n \to {\mathbb R}_{>0}$ by $\overline\gamma : {\mathbb R}^{n+1} \to {\mathbb R}_{\ge0}$. 
Let $X$ be a graph 
$$
X : D \to {\mathbb R}^3, \quad X(u_1, u_2)=(u_1, u_2, f(u_1, u_2))
$$
of a $C^{\infty}$ function $f:D (\subset {\mathbb R}^2) \to {\mathbb R}$. 
Set
$$
f_i:=f_{u_i}, \quad f_{ij}:=f_{u_iu_j}, \quad Df:=(f_1, f_2).
$$
Then,
$$
\Lambda=\frac{1}{2}\sum_{i, j=1, 2} \overline\gamma_{x_ix_j}\Big|_{(-Df, 1)}f_{u_iu_j}
$$
holds. In the special case where 
$\overline\gamma(Y)\equiv |Y|$, that is, $\gamma\equiv 1$, 
$$
\Lambda=\frac{(1+(f_2)^2) f_{11}-2f_1f_2f_{12}+(1+(f_1)^2)f_{22}}{2((f_1)^2+(f_2)^2+1)^{3/2}}
$$
holds. Hence, the equation ``$\Lambda \equiv {\rm constant}$'' is elliptic or hyperbolic depends on the $2 \times 2$ matrix $(\overline\gamma_{x_ix_j}\Big|_{(-Df, 1)})_{i, j=1, 2}$.
%
}\end{remark}

\begin{proposition}\label{CHR}
The anisotropic mean curvature of the Cahn-Hoffman map $\xi:S^n \to {\mathbb R}^{n+1}$ is $-1$ at any regular point $\nu \in S^n$ with respect to the unit normal $\nu$.  
Hence, particularly the anisotropic mean curvature of the Wulff shape (for the outward-pointing unit normal) is $-1$ at any regular point. 
\end{proposition}

\noindent {\it Proof}. \ 
Since $\xi^{-1}$ gives the unit normal vector field $\nu_{\xi}$ for the Cahn-Hoffman map $\xi:S^n \to {\mathbb R}^{n+1}$ (Corollary \ref{CCH2}), using (\ref{trace}), we have
\begin{equation}\label{tr}
\Lambda
=-\frac{1}{n}{\rm trace} (d(\xi\circ\xi^{-1}))
=
-\frac{1}{n}{\rm trace} (I_n)
=-1, 
\end{equation}
where $I_n$ is the $n\times n$ identity matrix. 
\hfill $\Box$

\begin{proposition}\label{cahn}
Set $A:=D^2\gamma+\gamma\cdot 1$. Take the unit normal vector $\tilde{\nu}$ of the 
Cahn-Hoffman map $\xi=\xi_\gamma$ so that $\tilde{\nu}=\nu$ holds at any point $\nu \in S^n$ where $\det A>0$ holds. 
Let $\{e_1, \cdots, e_n\}$ be an orthonormal basis of $T_\nu S^n$ such that $\{\tilde{\nu}, f_1, \cdots, f_n\}$ is compatible with the canonical orientation of ${\mathbb R}^{n+1}$, where $f_j:=(d_\nu\xi)(e_j)$. 
Then, at each point $\nu \in S^n$ satisfying $\det A<0$, $\tilde{\nu}=-\nu$ holds. 
If $\Lambda$ is the anisotropic mean curvature of $\xi$ with respect to $\tilde{\nu}$, then
\begin{equation}
\Lambda= \left \{
\begin{array}{l}
-1, \ (\det A>0),\\
1, \ (\det A<0).
\end{array}
\right.
\end{equation} 
\end{proposition}

Proposition \ref{cahn} gives

\begin{cor}\label{cor:cahn}
$\xi_\gamma$ is not CAMC in general.
\end{cor}

\begin{remark}{\rm
Example \ref{ex4'} gives a good example for Corollary \ref{cor:cahn}. The closed curve with self-intersection in Figure \ref{fig:ex4'} is the image $\xi_\gamma(S^1)$ of $\xi_\gamma$. At each regular point on each smooth arc containing a solid arc, $\Lambda=-1$, while at other regular points $\Lambda=1$. 
Hence, $\xi_\gamma(S^1)$ is not CAMC.
}
\end{remark}


\begin{proposition}\label{AMS}
There is no closed piecewise-smooth weakly immersed hypersurface in ${\mathbb R}^{n+1}$ that is a critical point of ${\cal F}_\gamma$. 
This means that there is no closed piecewise-smooth weakly immersed hypersurface in ${\mathbb R}^{n+1}$ whose anisotropic mean curvature is constant zero. 
\end{proposition}

\begin{proof}
Let $X:M^n \to {\mathbb R}^{n+1}$ be a piecewise-smooth weakly immersed hypersurface in ${\mathbb R}^{n+1}$. Consider the homotheties $X_\epsilon:=(1+\epsilon)X$ of $X$. Then,  
since ${\cal F}_\gamma(X_\epsilon)=(1+\epsilon)^n{\cal F}_\gamma(X)>0$, 
$$
\frac{d}{\;d\epsilon\;}\Big|_{\epsilon=0}{\cal F}_\gamma(X_\epsilon)=n{\cal F}_\gamma(X)>0
$$
holds. Hence $X$ is not a critical point of ${\cal F}_\gamma(X)$. 
\end{proof}

\section{Anisotropic parallel hypersurface and Steiner-type formula}\label{tube}

We assume that $\gamma:S^n \to {\mathbb R}_{>0}$ is of $C^2$, and use the same notations as in \S \ref{a-curv}. 
`{\it Anisotropic parallel hypersurface}' is a generalization of parallel hypersurface which is defined as follows. 

\begin{definition}[Anisotropic parallel hypersurface, cf. \cite{R}]\label{para}{\rm 
Let $X$ be a piecewise-$C^2$ weak immersion. 
For any real number $t$, we call the map $X_t:=X+t\tilde{\xi}:M \to {\mathbb R}^{n+1}$ the anisotropic parallel deformation of $X$ of height $t$. If $X_t$ is a piecewise-$C^2$ weak immersion, then we call it the anisotropic parallel hypersurface of $X$ of height $t$. 
}\end{definition}

The anisotropic energy ${\mathcal F}_\gamma(X_t)$ of the anisotropic parallel hypersurface $X_t:=X+t\tilde{\xi}$ is a polynomial of $t$ of degree at most $n$ as follows. 

\begin{theorem}[Steiner-type formula]\label{th-st}
Assume that $\gamma:S^n \to {\mathbb R}_{>0}$ is of $C^2$. 
Assume that $\displaystyle X:M=\cup_{i=1}^k M_i \rightarrow{\mathbb R}^{n+1}$ is a piecewise-$C^2$ weak immersion. 
Consider anisotropic parallel hypersurfaces $X_t=X+t\tilde{\xi}:M\setminus S(X) \to {\mathbb R}^{n+1}$, where $S(X)$ is the set of singular points of $X$. 
Then the following integral formula holds.
\begin{equation}\label{eq-st0}
{\mathcal F}_\gamma(X_t)
=
\int_M \gamma(\nu)\;dA_t
=
\int_M \gamma(\nu) (1-tk^\gamma_1)\cdots(1-t k^\gamma_n)\;dA,
\end{equation}
where $\int_M$ means $\sum_{i=1}^k \int_{M_i}$. 
Moreover, if the $r$-th anisotropic mean
curvature of $X$ for $\gamma$ is integrable on $M$ for $r=1, \cdots, n$, then 
\begin{equation}\label{eq-st}
{\mathcal F}_\gamma(X_t)
=
\sum_{r=0}^n (-1)^r t^r({}_nC_r)\int_M \gamma(\nu) H^\gamma_r\;dA 
\end{equation}
holds
\end{theorem}

Theorem \ref{th-st} is a direct consequence of Lemma \ref{shape2} (iii) below.

The isotropic version of Theorem \ref{th-st} is known as the Weyl's tube formula (\cite{W1939}). 
The isotropic $2$-dimensional version is the well-known Steiner's formula:

\begin{cor}[Steiner's formula]\label{cst}
Assume that $\displaystyle X:M \rightarrow{\mathbb R}^{3}$ is an immersion from a $2$-dimensional oriented connected compact $C^\infty$ manifold $M$ into ${\mathbb R}^{3}$ with unit normal $\nu:M \to S^2$. Denote by $H$, $K$ the mean and the Gaussian curvature of $X$, respectively. 
Then the area ${\mathcal A}(X_t)$ of the parallel surface $X_t=X+t\nu$ of height $t$ has the following representation. 
\begin{equation}
{\mathcal A}(X_t)
={\mathcal A}(X)-2t\int_MH \:dA+t^2\int_MK \;dA.
\end{equation}
\end{cor}

\vskip0.2truecm

Now assume that $X_t:=X+t\tilde{\xi}:M \to {\mathbb R}^{n+1}$ is an immersion. Let $(u_1, \cdots, u_n)$ be local coordinates in $M$. 
We compute the volume form
\begin{equation}\label{para1}
dA_t=(\det (g_{tij}))^{1/2}du^1\cdots du^n,
\end{equation}
where
$$
g_{tij}:=\langle (X_t)_i, (X_t)_j\rangle, \quad (X_t)_i:=(X_t)_{u_i}.
$$
We compute
\begin{equation}\label{para2}
g_{tij}=\langle X_{ti}, X_{tj}\rangle
=\langle X_i+t\tilde{\xi}_i, X_j+t\tilde{\xi}_j
\rangle
= \langle X_i, X_j\rangle
+ t (\langle X_i, \tilde{\xi}_j\rangle +\langle \tilde{\xi}_i, X_j\rangle)
+ t^2\langle \tilde{\xi}_i, \tilde{\xi}_j\rangle.
\end{equation}
By using (\ref{dxi3}), we have
\begin{equation}\label{xi}
\langle \tilde{\xi}_i, \tilde{\xi}_j\rangle
= \langle -\tilde{h}_{il}g^{lk}X_k, -\tilde{h}_{ja}g^{ab}X_b\rangle
= \tilde{h}_{il}g^{lk}\tilde{h}_{ja}g^{ab}g_{kb}
=\tilde{h}_{il}g^{lk}\tilde{h}_{jk}.
\end{equation}
(\ref{para2}) with (\ref{dxi2}) and (\ref{xi}) gives
\begin{equation}\label{para3}
g_{tij}=g_{ij}-t(\tilde{h}_{ij}+\tilde{h}_{ji})+t^2 \tilde{h}_{il}g^{lk}\tilde{h}_{jk}.
\end{equation}
Define three $n\times n$ matrices as follows:
$$
G:=(g_{ij}), \ \tilde{B}:=(\tilde{h}_{ij}), \ G_t:=(g_{tij}).
$$
Note that $\tilde{B}$ is not symmetric in general.

From (\ref{para3}), we have
\begin{eqnarray}
\det G_t&=&
\det G \cdot
\det \Bigl(
I_n-t(\tilde{B}G^{-1}+{}^t\!\tilde{B}G^{-1})+t^2 (\tilde{B}G^{-1}{}^t\!\tilde{B}G^{-1}),
\Bigr) \label{det1}\\
&=&
\det G \cdot
\det (I_n-t\tilde{B}G^{-1})
\cdot \det(I_n-t\;{}^t\!\tilde{B}G^{-1}), \label{det2}
\end{eqnarray}
here ${}^t\!\tilde{B}$ is the transposed matrix of $\tilde{B}$. 

\begin{lemma}\label{shape2}
(i) 
\begin{equation}\label{shape3}
\det (I_n-t\tilde{B}G^{-1})
=\det(I_n-t\;{}^t\!\tilde{B}G^{-1}). 
\end{equation}
Hence, $\det (I_n-t\tilde{B}G^{-1})$ and $\det(I_n-t\;{}^t\!\tilde{B}G^{-1})$ have the same eigenvalues.

(ii) 
\begin{eqnarray}
\det (I_n-t\tilde{B}G^{-1})
&=&\det(I_n-t\;{}^t\!\tilde{B}G^{-1})\\
&=&
(1-tk^\gamma_1)\cdots(1-t k^\gamma_n) \\
&=&\sum_{r=0}^n (-1)^r t^r \sum_{1 \le l_1< \cdots l_r\le n}
k^\gamma_{l_1}\cdots k^\gamma_{l_r}\\
&=& \sum_{r=0}^n (-1)^r t^r \sigma^\gamma_r 
= \sum_{r=0}^n (-1)^r t^r({}_nC_r) H^\gamma_r.
\end{eqnarray}

(iii)
\begin{eqnarray}
dA_t
&=&(1-tk^\gamma_1)\cdots(1-t k^\gamma_n)\;dA \nonumber\\
&=&
\sum_{r=0}^n (-1)^r t^r({}_nC_r) H^\gamma_r\;dA. \nonumber
\end{eqnarray}
\end{lemma}

\noindent{\it Proof.} \ 
First we prove (i). 
Observe
\begin{eqnarray}
\det (I_n-t\tilde{B}G^{-1})
&=& \det \;{}^t\!(I_n-t\tilde{B}G^{-1})
= \det (\;{}^t\!I_n-t\;{}^t\!G^{-1}\;{}^t\!\tilde{B})\\
&=& \det (I_n-tG^{-1}\;{}^t\!\tilde{B})
= \det G\cdot\det (I_n-tG^{-1}\;{}^t\!\tilde{B})\cdot\det G^{-1}\\
&=&\det(I_n-t\;{}^t\!\tilde{B}G^{-1}),
\end{eqnarray}
which gives (i). 

(ii) is given by (\ref{exp3}) by replacing $\tau$ with $-t$ because 
$S^\gamma=
(\tilde{h}_{ij})(g^{ij})=\tilde{B}G^{-1}$ (Lemma \ref{shapeL}). 

(iii) is given by (\ref{para1}), (\ref{det2}) and (ii) of this lemma. 
\hfill $\Box$

\section{Minkowski-type formula}\label{min}

In this section we give another integral formula that is a generalization of the Minkowski formula
(see \cite{HL2008} for smooth case). 

\begin{theorem}[Minkowski-type formula]\label{th-min}
Assume that $\gamma:S^n\to {\mathbb R}_{>0}$ is of $C^2$. 
Assume that $\displaystyle X:M=\cup_{i=1}^k M_i \rightarrow{\mathbb R}^{n+1}$ is a closed piecewise-$C^2$ weak immersion whose $r$-th anisotropic mean
curvature for $\gamma$ is integrable on $M$ for $r=1, \cdots, n$. 
Assume also that $X$ satisfies the following condition.

$$
\tilde{\xi}_i(\zeta)=\tilde{\xi}_j(\zeta), \quad \forall \zeta \in \partial M_i\cap \partial M_j, 
$$
here
$\tilde{\xi}_i:=\xi\circ\nu_i :M_i \to {\mathbb R}^{n+1}$ is the Cahn-Hoffman field along $X|_{M_i}$, and the tangent space of a submanifold of ${\mathbb R}^{n+1}$ is naturally identified with a linear subspace of ${\mathbb R}^{n+1}$. 

Then the following integral formula holds.
\begin{equation}\label{eq-min}
\int_M (\gamma(\nu) H^\gamma_r + \langle X, \nu\rangle  H^\gamma_{r+1})\;dA
=0, \quad r=0, \cdots, n-1,
\end{equation}
where $H^\gamma_r$ is the $r$-th anisotropic mean curvature of $X$ (Definition \ref{curv}). 
In the special case where $r=0$, we have
\begin{equation}\label{eq-min1}
\int_M (\gamma(\nu)+\Lambda \langle X, \nu\rangle) \;dA=0.
\end{equation}
\end{theorem}

Before proving Theorem \ref{th-min}, we give a useful lemma:

\begin{lemma}\label{lem-V}
Assume the same assumption as Theorem \ref{th-min}. 
For $t \in {\mathbb R}$, $|t|<<1$, set
$$X_t:=X+t\tilde{\xi}.$$
Then, 
\begin{eqnarray}
\frac{d}{dt}V(X_t)&=&\int_M\gamma(\nu) \;dA_t
={\mathcal F}_\gamma(X_t) \label{dv0}\\
&=&\sum_{r=0}^n (-1)^r ({}_nC_r)t^r\int_M \gamma(\nu)  H^\gamma_r\;dA \label{dv}
\end{eqnarray}
holds.
\end{lemma}

\noindent {\it Proof.} \ 
The unit normal along $X_t$ is the same as that of $X$. Hence, using (\ref{firstv}) and Theorem \ref{th-st}, we obtain
\begin{eqnarray}
\frac{d}{dt}V(X_t)&=&\int_M \langle \tilde{\xi}, \nu\rangle \;dA_t
=\int_M \langle D\gamma|_\nu+\gamma(\nu)\nu, \nu\rangle \;dA_t \nonumber\\
&=&\int_M\gamma(\nu) \;dA_t
={\mathcal F}_\gamma(X_t) \nonumber\\
&=&\sum_{r=0}^n (-1)^r ({}_nC_r)t^r\int_M \gamma(\nu)  H^\gamma_r\;dA. \nonumber
\end{eqnarray}

\hfill $\Box$

\noindent{\it Proof of Theorem \ref{th-min}.}\ 
For $t \in {\mathbb R}$, $|t|<<1$, set
$$X_t:=X+t\tilde{\xi}.$$
Since the unit normal along $X_t$ is the same as that of $X$, it holds that 
\begin{equation}
V(X_t)=\frac{1}{n+1}
\int_M\langle X_t, \nu\rangle \;dA_t
=\frac{1}{n+1}
\int_M\langle X+t\tilde{\xi}, \nu\rangle \;dA_t.
\end{equation}
Hence we have
\begin{eqnarray}
\frac{d}{dt}V(X_t) &=&
\frac{1}{n+1}\int_M\langle\tilde{\xi}, \nu\rangle \;dA_t
+ \frac{1}{n+1}
\int_M\langle X+t\tilde{\xi}, \nu\rangle \;\frac{\partial(dA_t)}{\partial t} \label{III}\\
&=:& I+II. \label{I-II}
\end{eqnarray}
We compute, using Lemma \ref{shape2} (iii), to obtain
\begin{eqnarray}
II&:=&
\frac{1}{n+1}
\int_M\langle X+t\tilde{\xi}, \nu\rangle \;\frac{\partial(dA_t)}{\partial t}\nonumber\\
&=&
\frac{1}{n+1}
\int_M (\langle X, \nu\rangle
+t\gamma(\nu)) \;\frac{\partial}{\partial t}\sum_{r=0}^n (-1)^r t^r({}_nC_r) H^\gamma_r\;dA \nonumber\\
&=&
\frac{1}{n+1}
\int_M (\langle X, \nu\rangle
+t\gamma(\nu)) \sum_{r=1}^n (-1)^r rt^{r-1}({}_nC_r) H^\gamma_r\;dA \nonumber\\
&=&
\frac{1}{n+1}
\sum_{r=1}^n (-1)^r rt^{r-1}({}_nC_r)
\int_M \langle X, \nu\rangle  H^\gamma_r\;dA \nonumber\\
&&+
\frac{1}{n+1}
\sum_{r=1}^n (-1)^r rt^r({}_nC_r)
\int_M 
\gamma(\nu)  H^\gamma_r\;dA \nonumber\\
&=&
\frac{1}{n+1}
\sum_{r=0}^{n-1} (-1)^{r+1} (r+1)t^r({}_nC_{r+1})
\int_M \langle X, \nu\rangle  H^\gamma_{r+1}\;dA \nonumber\\
&&+
\frac{1}{n+1}
\sum_{r=1}^n (-1)^r rt^r({}_nC_r)
\int_M 
\gamma(\nu) H^\gamma_r\;dA.\label{II}
\end{eqnarray}
Again, by using Lemma \ref{shape2} (iii), we obtain
\begin{eqnarray}
I&:=&\frac{1}{n+1}\int_M\langle\tilde{\xi}, \nu\rangle \;dA_t
=\frac{1}{n+1}\int_M\langle D\gamma+\gamma \nu, \nu\rangle \;dA_t \nonumber\\
&=&
\frac{1}{n+1}\int_M \gamma \;dA_t
=\frac{1}{n+1}
\sum_{r=0}^n (-1)^r ({}_nC_r)t^r\int_M \gamma(\nu)  H^\gamma_r\;dA. \label{eqI}
\end{eqnarray}
On the other hand, by Lemma \ref{lem-V}, we have
\begin{equation}\label{dvdt}
\frac{d}{dt}V(X_t)
=\sum_{r=0}^n (-1)^r ({}_nC_r)t^r\int_M \gamma(\nu)  H^\gamma_r\;dA.
\end{equation}
The equalities (\ref{I-II}), (\ref{II}),  
 (\ref{eqI}) and (\ref{dvdt}) give
\begin{eqnarray}
\sum_{r=0}^n (-1)^r ({}_nC_r)t^r\int_M \gamma(\nu)  H^\gamma_r\;dA
&=&
\frac{1}{n+1}
\sum_{r=0}^n (-1)^r ({}_nC_r)t^r\int_M \gamma(\nu)  H^\gamma_r\;dA \nonumber\\
&&+
\frac{1}{n+1}
\sum_{r=0}^{n-1} (-1)^{r+1} (r+1)t^r({}_nC_{r+1})
\int_M \langle X, \nu\rangle  H^\gamma_{r+1}\;dA \nonumber\\
&&+
\frac{1}{n+1}
\sum_{r=1}^n (-1)^r rt^r({}_nC_r)
\int_M 
\gamma(\nu)  H^\gamma_r\;dA,  \label{f-min}
\end{eqnarray}
that is,
\begin{equation} \label{f-min2}
\sum_{r=0}^n (-1)^r (n-r)({}_nC_r)t^r\int_M \gamma(\nu)  H^\gamma_r\;dA
=
-\sum_{r=0}^{n-1} (-1)^r (r+1)t^r({}_nC_{r+1})
\int_M \langle X, \nu\rangle  H^\gamma_{r+1}\;dA, 
\end{equation}
which gives
\begin{equation} \label{f-min3}
\sum_{r=0}^{n-1} (-1)^r (n-r)({}_nC_r)t^r\int_M (\gamma(\nu)  H^\gamma_r + \langle X, \nu\rangle  H^\gamma_{r+1})\;dA
=0.
\end{equation}
(\ref{f-min3}) implies
\begin{equation}\label{f-min4}
\int_M (\gamma(\nu)  H^\gamma_r + \langle X, \nu\rangle  H^\gamma_{r+1})\;dA
=0, \quad r=0, \cdots, n-1,
\end{equation}
which is the desired result. 
\hfill $\Box$

\begin{remark}
If we want to have only the formula (\ref{eq-min1}), the following proof works. 

For $t \in {\mathbb R}$, $|t|<<1$, set
$$
X_t:=(1+t)X.
$$
Then we have
$$
f(t):={\mathcal F}_\gamma(X_t)
=\int_M\gamma(\nu)(1+t)^n\;dA,
$$
and 
\begin{equation}\label{eqa}
f'(t)=n(1+t)^{n-1}\int_M\gamma(\nu)\;dA.
\end{equation}
Hence
\begin{equation}\label{eqb}
f'(0)=n\int_M\gamma(\nu)\;dA
\end{equation}
holds. On the other hand, the first variation formula (Lemma \ref{firstFF}) gives
\begin{equation}\label{eqc}
f'(0)=-n\int_M \Lambda \Bigl\langle
\frac{\partial X_t}{\partial t}\Big|_{t=0}, 
\nu \Bigr\rangle \;dA
=-n\int_M \Lambda \langle
X, 
\nu \rangle \;dA.
\end{equation}
(\ref{eqb}) with (\ref{eqc}) gives (\ref{eq-min1}). 
\end{remark}

\section{Proof of Theorem \ref{PK2017}}\label{s-unique}

First we give an outline of the proof of Theorem \ref{PK2017}. 
The idea is to generalize the proof of the uniqueness of stable closed CMC hypersurface in ${\mathbb R}^{n+1}$ given 
by \cite{HW1991}, which was used also in \cite{P1998} and \cite{P2012}.  
In order to prove Theorem \ref{PK2017}, we first recall that the Cahn-Hoffman field $\tilde{\xi}$ along a closed CAMC hypersurface $X:M=\cup_{i=1}^k M_i \to {\mathbb R}^{n+1}$ is well-defined on the whole of $M$ (Theorem \ref{CHP}). 
Then, we consider 
the anisotropic parallel hypersurfaces $X_t:=X+t\tilde{\xi}$, ($t \in {\mathbb R}$, $|t|<<1$) of $X$. 
By taking homotheties of $X_t$ if necessary, we have a volume-preserving variation $Y_t=\mu(t)X_t=\mu(t)(X+t\tilde{\xi})$, ($\mu(t) >0$, $\mu(0)=1$) of $X$. 
By using the Steiner-type formula (Theorem \ref{th-st}) and the Minkowski-type formula (Theorem \ref{th-min}), we prove that 
$\displaystyle 
\frac{d^2{\mathcal F}_\gamma(Y_t)}{dt^2}\Big|_{t=0}
=\frac{-1}{n}\int_M \gamma(\nu) 
\sum_{1 \le i < j\le n}(k_i^\gamma-k_j^\gamma)^2\; dA
$
holds, where $k_i^\gamma$ are the anisotropic principal curvatures of $X$ (Definition \ref{curv}). 
Since $\gamma$ is convex, all $k_i^\gamma$ are real values on $M\setminus S(X)$ (Lemma \ref{HL}). Hence, if $X$ has constant anisotropic mean curvature $\Lambda$ and stable, $k_1^\gamma=\cdots =k_n^\gamma=\Lambda/n \ne 0$ holds 
on $M\setminus S(X)$, from which we can prove that $X(M) = (1/|\Lambda|)W$
holds. 

\begin{remark}\label{uniqueBP2}{\rm 
B. Palmer \cite{P2012} proved the same conclusion as our Theorem \ref{PK2017} under some extra assumptions. Actually, under the following assumptions (i) - (vii), he proved that $X(M)$ coincides with $W_\gamma$ (up to homothety and translation).

(i) $n=2$.

(ii) $\gamma:S^2 \to {\mathbb R}_{>0}$ is of $C^3$ and convex.

(iii) Wulff shape $W_\gamma$ is piecewise smooth, and there exists a positive value $c$ such that all of the principal curvatures of $W_\gamma$ with respect to the inward-pointing normal are bounded from below by $c$. 

(iv) $X:M\to {\mathbb R}^3$ is a closed piecewise sooth surface without self-intersection with unit normal $\nu:M\setminus {\cal S}(X) \to S^2$, where ${\cal S}(X)$ is the set af all singular points of $X$. 

(v) The Cahn-Hoffman field $\tilde{\xi}:M\setminus {\cal S}(X) \to {\mathbb R}^3$ can be continuously extended to the whole of $M$. 

(vi) $\nu(M\setminus {\cal S}(X))$ is included in the set of all unit normals to $W_\gamma$. 

(vii) $X$ is stable. 
}\end{remark}


\noindent {\it Proof of Theorem \ref{PK2017}.} \ 
Assume that $X:M=\cup_{i=1}^k M_i \to {\mathbb R}^{n+1}$ is a closed piecewise-$C^2$ CAMC hypersurface for $\gamma:S^n \to {\mathbb R}_{>0}$. 
Because of Lemma \ref{CHP}, the Cahn-Hoffman field $\tilde{\xi}:M \to {\mathbb R}^{n+1}$ along X can be defined on $M$. 
Consider the anisotropic parallel hypersurfaces $X_t:M \to {\mathbb R}^{n+1}$ of $X$, that is
$$
X_t:=X+t\tilde{\xi}, \quad t \in {\mathbb R}, |t|<<1.
$$
By taking homotheties of $X_t$ if necessary, we have a volume-preserving variation $Y_t$ of $X$ which is represented as follows.
$$
Y_t=\mu(t)X_t=\mu(t)(X+t\tilde{\xi}), \quad \mu(t) >0, \ \mu(0)=1.
$$
Set
$$
F_0:={\mathcal F}_\gamma(X), \quad V_0:=V(X).
$$
And set
$$
f(t):={\mathcal F}_\gamma(Y_t), \quad v(t):=V(Y_t). 
$$
Then,
\begin{equation}\label{fv}
f(t)=(\mu(t))^n {\mathcal F}_\gamma(X_t), \quad 
v(t)=(\mu(t))^{n+1} V(X_t), \quad f(0)=F_0, \quad v(0)=V_0.
\end{equation}
We will compute $f''(0)$. 
In order to do it, we will compute $\mu'(0)$ and $\mu''(0)$ by using $v(t)\equiv V_0$.  

From (\ref{fv}), we have
\begin{equation}\label{v'}
v'(t)=(n+1)(\mu(t))^n\mu'(t)V(X_t)
+(\mu(t))^{n+1}\frac{dV(X_t)}{dt},
\end{equation}
\begin{eqnarray}
v''(t)&=&(n+1)n(\mu(t))^{n-1}(\mu'(t))^2V(X_t)
+(n+1)(\mu(t))^n\mu''(t)V(X_t) \nonumber\\
&&+2(n+1)(\mu(t))^n\mu'(t)\frac{dV(X_t)}{dt}
+(\mu(t))^{n+1}\frac{d^2V(X_t)}{dt^2}.\label{v''}
\end{eqnarray}
From (\ref{dv0}), 
\begin{equation}\label{dv00}
\frac{dV(X_t)}{dt}\Big|_{t=0}=F_0.
\end{equation}
Using $v'(0)=0$, (\ref{v'}) with (\ref{dv00}) gives
\begin{equation}\label{mu1}
\mu'(0)=-\frac{F_0}{(n+1)V_0}.
\end{equation}
Using (\ref{dv}), we obtain
\begin{equation}\label{v''0}
\frac{d^2V(X_t)}{dt^2}\Big|_{t=0}
=-({}_nC_1)\int_M\gamma\Lambda\;dA
=-n\Lambda F_0.
\end{equation}
From (\ref{eq-min1}) in Theorem \ref{th-min}, we have
$$
F_0+\Lambda(n+1)V_0=0,
$$
and hence
\begin{equation}\label{lam}
\Lambda=\frac{-F_0}{(n+1)V_0}.
\end{equation}
Using (\ref{v''}), (\ref{dv00}), (\ref{mu1}), (\ref{v''0}) and (\ref{lam}), we obtain
\begin{equation}\label{v''1}
0=v''(0)
=\frac{-2}{n+1}\cdot\frac{F_0^2}{V_0}+(n+1)V_0\mu''(0),
\end{equation}
which gives
\begin{equation}\label{mu2}
\mu''(0)
=\frac{2}{(n+1)^2}\cdot\frac{F_0^2}{V_0^2}.
\end{equation}

From (\ref{fv}), we have
\begin{equation}\label{f'}
f'(t)=n(\mu(t))^{n-1}\mu'(t){\mathcal F}_\gamma(X_t)
+(\mu(t))^n\frac{d{\mathcal F}_\gamma(X_t)}{dt},
\end{equation}
\begin{eqnarray}
f''(t)&=&n(n-1)(\mu(t))^{n-2}(\mu'(t))^2{\mathcal F}_\gamma(X_t)
+n(\mu(t))^{n-1}\mu''(t){\mathcal F}_\gamma(X_t) \nonumber\\
&&+2n(\mu(t))^{n-1}\mu'(t)\frac{d{\mathcal F}_\gamma(X_t)}{dt}
+(\mu(t))^n\frac{d^2{\mathcal F}_\gamma(X_t)}{dt^2}.\label{f''}
\end{eqnarray}

Using (\ref{eq-st}) in Theorem \ref{th-st}, (\ref{mu1}), (\ref{lam}), (\ref{mu2}) and (\ref{f''}), we obtain
\begin{eqnarray}
f''(0)&=&
\frac{-n(n-1)}{(n+1)^2}\cdot\frac{F_0^3}{V_0^2}+2({}_nC_2)\int_M\gamma(\nu)H_2^\gamma\;dA \nonumber\\
&=& -n(n-1)\Lambda^2F_0+2\int_M \gamma(\nu) \sum_{1 \le i < j\le n}k_i^\gamma k_j^\gamma\;dA \nonumber\\
&=&
\frac{-(n-1)}{n}\int_M \gamma(\nu)
\Bigl[
(k_1^\gamma+\cdots + k_n^\gamma)^2
-\frac{2n}{n-1}
\sum_{1 \le i < j\le n}k_i^\gamma k_j^\gamma
\Bigr]
\;dA \nonumber\\
&=&
\frac{-1}{n}\int_M \gamma(\nu) 
\sum_{1 \le i < j\le n}(k_i^\gamma-k_j^\gamma)^2\; dA. \label{f''0}
\end{eqnarray}
Note that, since $\gamma$ is convex, by Lemma \ref{HL}, all $k_i^\gamma$ are real values on $M\setminus S(X)$. If $X$ is stable, then $f''(0)$ must be non-negative, and so  (\ref{f''0}) implies that
\begin{equation}
k_1^\gamma=\cdots =k_n^\gamma
\end{equation}
holds at each point in $M\setminus S(X)$. Since $\Lambda=(k_1^\gamma+\cdots +k_n^\gamma)/n$ is constant on $M\setminus S(X)$,
\begin{equation}
k_1^\gamma=\cdots =k_n^\gamma=\Lambda/n
\end{equation}
holds on $M\setminus S(X)$. Therefore, from Corollary 1 in \cite{R}, 
\begin{equation}\label{rei}
X(M\setminus S(X)) \subset rW
\end{equation}
holds for some $r>0$. Because $M$ is closed and $W$ has anisotropic mean curvature $-1$, (\ref{rei}) implies 
$$
X(M) = (1/|\Lambda|)W.
$$
\hfill $\Box$

\section{Proof of Proposition \ref{FV}}\label{PFV}

For simplicity, we write $M$ instead of $M_0$. 
We take an exhaustion $\{M_{t}\}_{0\le t<1}$ of $M^o$. That is, 

(i) $M_{t} \subset M^o$, $\forall t \in [0, 1)$, and

(ii) For any compact set $K\subset M^o$, there exists a number $t_0 \in  (0, 1)$ such that $M_t \supset K$ for all $t \in (t_0, 1)$.

\noindent Set $M_1:=M$. 

Let 
$X_\epsilon:M\to {\mathbb R}^{n+1}$, ($\epsilon \in J:=[-\epsilon_0, \epsilon_0]$, $\epsilon_0>0$), be a variation of $X$. Set $X_\epsilon^t:=X_\epsilon|_{M_t}$. 
And set
$$
\psi_\epsilon:=\langle \frac{\partial X_\epsilon}{\partial\epsilon}, \nu_\epsilon\rangle.
$$
Denote by $\Lambda_\epsilon$ the anisotropic mean curvature of $X_\epsilon$, and by $dA_\epsilon$ the volume form of $M$ with the metric induced by $X_\epsilon$. Note that $\Lambda_\epsilon$ is defined at regular points of $X_\epsilon$. Also note that, if $X_\epsilon$ has singular points in $\partial M$, then $dA_\epsilon$ is singular but it is well-defined at these points. From Lemma \ref{firstFF}, we have
\begin{equation} \label{FV0}
\delta {\mathcal F_\gamma}^t :=
\frac{d{\mathcal F}_\gamma(X_\epsilon^t)}{d\epsilon}\Big|_{\epsilon=0}
=
- \int_{M_t} n\Lambda\psi\;dA
-\oint_{\partial {M_t}} \langle \delta X, R(p(\tilde{\xi}))\rangle\;d{\tilde s}=:I(t)+I\!I(t)
\end{equation}
for $t \in [0, 1)$. 
For $\epsilon \in [-\epsilon_0, 0)\cup (0, \epsilon_0]$ and $t \in [0, 1]$, set
\begin{equation}\label{ff}
f(\epsilon, t):=\int_{M_t}\frac{1}{\epsilon}\Bigl(
\gamma(\nu_\epsilon)\;dA_\epsilon - \gamma(\nu) \;dA
\Bigr).
\end{equation}
Then, (\ref{FV0}) implies that
\begin{equation}\label{ff2}
\lim_{\epsilon \to 0} f(\epsilon, t)=
\frac{d{\mathcal F}_\gamma(X_\epsilon^t)}{d\epsilon}\Big|_{\epsilon=0}
=I(t)+I\!I(t), \quad t \in [0, 1)
\end{equation}
holds. 

We also have the following. For any $\epsilon \in (-\epsilon_0, 0)\cup (0, \epsilon_0)$ and $t \in (0, 1)$, there exists a number $\theta \in (0, 1)$ such that
\begin{equation} \label{FV10}
f(\epsilon, t)=
\frac{d{\mathcal F}_\gamma(X_\sigma^t)}{d\sigma}\Big|_{\sigma=\theta\epsilon}
=
- \int_{M_t} n\Lambda_{\theta\epsilon}\psi_{\theta\epsilon}\;dA_{\theta\epsilon}
-\oint_{\partial {M_t}} \langle \frac{\partial X_\sigma}{\partial\sigma}\Big|_{\sigma=\theta\epsilon}, R(p(\tilde{\xi}_{\theta\epsilon}))\rangle\;d{\tilde s}_{\theta\epsilon}=:I_{\theta\epsilon}(t)+I\!I_{\theta\epsilon}(t)
\end{equation}
holds. 
Set
\begin{equation}\label{C1}
c_1(\epsilon)=\sup_{M^0} |\Lambda_{\theta\epsilon}\psi_{\theta\epsilon}-\Lambda_0\psi_0|,
\end{equation}
\begin{equation}\label{C2}
c_2=\sup_{M^0} |\Lambda_0\psi_0|.
\end{equation}
Because $\Lambda_\epsilon$ is bounded on $M^0$ for any $\epsilon$, $c_1(\epsilon)$ and $c_2$ are finite numbers. 

In order to prove Proposition \ref{FV}, we prepare several claims. 

\begin{claim}\label{c1}
The convergence in (\ref{ff2}) is uniform with respect to $t$.
\end{claim}

\noindent {\it Proof.} \ 
We compute
\begin{eqnarray}
\frac{1}{n}|I_{\theta\epsilon}(t)-I(t)|
&=&
\Bigl|
\int_{M_t}(\Lambda_{\theta\epsilon}\psi_{\theta\epsilon}dA_{\theta\epsilon}
-\Lambda_0\psi_0dA_0)
\Bigr|\\
&\le & 
\int_{M_t}|\Lambda_{\theta\epsilon}\psi_{\theta\epsilon}dA_{\theta\epsilon}
-\Lambda_0\psi_0dA_0|\\
&\le& 
\int_{M_t}|\Lambda_{\theta\epsilon}\psi_{\theta\epsilon}-\Lambda_0\psi_0|dA_{\theta\epsilon}
+
\int_{M_t}|\Lambda_0\psi_0|\cdot |dA_{\theta\epsilon}-dA_0|\\
&\le&
c_1(\epsilon) 
\int_{M_t}dA_{\theta\epsilon}
+
c_2\int_{M_t}|dA_{\theta\epsilon}-dA_0|\\
&\le&
c_1(\epsilon) 
\int_{M}dA_{\theta\epsilon}
+
c_2\int_{M}|dA_{\theta\epsilon}-dA_0|.
\end{eqnarray}
Since $\displaystyle \int_{M}dA_{\theta\epsilon}$ is bounded and
$$
\lim_{\epsilon \to 0}c_1(\epsilon)=0, \quad 
\lim_{\epsilon \to 0}\int_{M}|dA_{\theta\epsilon}-dA_0|=0
$$
hold, 
$$
|I_{\theta\epsilon}(t)-I(t)| \rightarrow 0, \ \textup{as} \ \epsilon \rightarrow 0
$$
uniformly with respect to $t$. 
Next we have 
\begin{eqnarray}
\frac{1}{n}|I\!I_{\theta\epsilon}(t)-I\!I(t)|
&=& 
\Bigl|\oint_{\partial {M_t}} \langle \delta X|_{\theta\epsilon}, R(p(\tilde{\xi}_{\theta\epsilon}))\rangle\;d{\tilde s}_{\theta\epsilon}
-
\oint_{\partial {M_t}} \langle \delta X|_0, R(p(\tilde{\xi}_0))\rangle\;d{\tilde s}_0
\Bigr|\\
&\le&
\oint_{\partial {M_t}} \Bigl|\langle \delta X|_{\theta\epsilon}, R(p(\tilde{\xi}_{\theta\epsilon}))\rangle\;d{\tilde s}_{\theta\epsilon}
-
\langle \delta X|_0, R(p(\tilde{\xi}_0))\rangle\;d{\tilde s}_0
\Bigr|\\
&\le&
\oint_{\partial {M_t}} \Bigl[\Bigl|\langle \delta X|_{\theta\epsilon}, R(p(\tilde{\xi}_{\theta\epsilon}))\rangle
-
\langle \delta X|_0, R(p(\tilde{\xi}_0))\rangle\Bigr|\;d{\tilde s}_{\theta\epsilon}
\\
&&+
\large|
\langle \delta X|_0, R(p(\tilde{\xi}_0))\rangle\large|\cdot\large|\;d{\tilde s}_{\theta\epsilon}-d{\tilde s}_0
\large|\Bigr]\\
&\le &
c_4(\epsilon)
\oint_{\partial {M_t}} \;d{\tilde s}_{\theta\epsilon}
+c_5
\oint_{\partial {M_t}}
\large|\;d{\tilde s}_{\theta\epsilon}-d{\tilde s}_0
\large|,
\end{eqnarray}
where
\begin{equation}\label{bd1}
c_4(\epsilon)
:=
\max_{0\le t \le 1}
\max_{\partial M_t} \Bigl|\langle \delta X|_{\theta\epsilon}, R(p(\tilde{\xi}_{\theta\epsilon}))\rangle
-
\langle \delta X|_0, R(p(\tilde{\xi}_0))\rangle\Bigr|,
\end{equation}
\begin{equation}\label{bd2}
c_5=
\max_{0\le t \le 1}
\max_{\partial M_t} \Bigl|\langle \delta X|_0, R(p(\tilde{\xi}_0))\rangle\Bigr|.
\end{equation}
Since $\displaystyle \Bigl\{
\oint_{\partial {M_t}} \;d{\tilde s}_{\theta\epsilon} \;;\; t \in [0, 1], \ \epsilon\in [-\epsilon_0, \epsilon_0]
\Bigr\}$ is bounded, and since
$$
\lim_{\epsilon \to 0}c_4(\epsilon)=0, \quad 
\lim_{\epsilon \to 0}\max_{0 \le t \le 1}\oint_{\partial {M_t}}
\large|\;d{\tilde s}_{\theta\epsilon}-d{\tilde s}_0
\large|=0
$$
hold, 
$$
|I\!I_{\theta\epsilon}(t)-I\!I(t)| \rightarrow 0, \ \textup{as} \ \epsilon \rightarrow 0
$$
uniformly with respect to $t$. 
Therefore, the convergence in (\ref{ff2}) is uniform with respect to $t$.
\hfill $\Box$

\vskip0.5truecm

\begin{claim}\label{c3}
The improper integral $\displaystyle \int_M |\Lambda| \; dA$ converges. 
\end{claim}

\noindent {\it Proof.} \ 
By assumption, the anisotropic mean curvature $\Lambda$ of $X$ is bounded on $M$. Hence, for any exhaustion $\{\hat{M}_{s}\}_{0\le s<1}$ of $M$ (for the definition of exhaustion, see the beginning of \S \ref{PFV}), $\displaystyle \int_{\hat{M}_s} |\Lambda| \; dA$ is bounded and an increasing function of $s$. Therefore, 
$\displaystyle \lim_{s\to 1-0} \int_{\hat{M}_s} |\Lambda| \; dA$ converges, which implies the desired result. 
\hfill $\Box$

\vskip0.5truecm

\begin{claim}\label{c2}
$\lim_{t \to 1-0} I(t)$ and $\lim_{t \to 1-0} I\!I(t)$ converge. Hence 
$\lim_{t \to 1-0}\Bigl(
\lim_{\epsilon \to 0} f(\epsilon, t)
\Bigr)$ converges. And we may write
\begin{equation}\label{ff5}
\lim_{t \to 1-0}\Bigl(
\lim_{\epsilon \to 0} f(\epsilon, t)
\Bigr)
=
- \int_{M} n\Lambda\psi\;dA
-\oint_{\partial {M}} \langle \delta X, R(p(\tilde{\xi}))\rangle\;d{\tilde s}.
\end{equation}
\end{claim}

\noindent {\it Proof.} \ 
Recall
$$
I(t)=- \int_{M_t} n\Lambda\psi\;dA, \quad 
I\!I(t)=
-\oint_{\partial {M_t}} \langle \delta X, R(p(\tilde{\xi}))\rangle\;d{\tilde s}.
$$
For $0<t_1<t_2<1$, we have
\begin{eqnarray}
\frac{1}{n}|I(t_2)-I(t_1)|
&=&
\Bigl|\int_{M_{t_2}\setminus M_{t_1}} \Lambda\psi\;dA
\Bigr|
\le 
\int_{M_{t_2}\setminus M_{t_1}} \bigl|\Lambda\psi\bigr|\;dA \nonumber \\
&\le& 
\max_M|\psi|\cdot \int_{M_{t_2}\setminus M_{t_1}} \bigl|\Lambda\bigr|\;dA. \label{in1}
\end{eqnarray}
Since $\displaystyle \int_M |\Lambda| \; dA$ converges (Claim \ref{c3}), (\ref{in1}) implies that, 
for any $\sigma>0$, there exists a number $T \in (0, 1)$ such that if $T<t_1<t_2<1$, then $|I(t_2)-I(t_1)|<\sigma$. 
Hence $\lim_{t \to 1-0} I(t)$ converges. Convergence of $\lim_{t \to 1-0} I\!I(t)$ is obvious.
\hfill $\Box$

\vskip0.5truecm

The following claim is obviously true. 

\begin{claim}\label{c4}
For any $\epsilon \in (-\epsilon_0, 0)\cup (0, \epsilon_0)$, 
\begin{equation}\label{ff3}
\lim_{t \to 1-0} f(\epsilon, t)=
\int_{M}\frac{1}{\epsilon}\Bigl(
\gamma(\nu_\epsilon)\;dA_\epsilon - \gamma(\nu) \;dA
\Bigr)
\end{equation}
holds. 
\end{claim}

\noindent {\it Proof of Proposition \ref{FV}.} \ 
From Claims \ref{c1}, \ref{c2}, and \ref{c4}, the limit
$$
\lim_{\substack{\epsilon \to 0 \\ t \to 1-0}} f(\epsilon, t)
$$
exists, and 
\begin{equation}\label{ff4}
\lim_{\epsilon \to 0}\Bigl(\lim_{t \to 1-0} f(\epsilon, t)\Bigr)=
\lim_{\substack{\epsilon \to 0 \\ t \to 1-0}} f(\epsilon, t)
=
\lim_{t \to 1-0}\Bigl(\lim_{\epsilon \to 0} f(\epsilon, t)\Bigr)
\end{equation}
holds. 
(\ref{ff4}) together with (\ref{ff3}), (\ref{ff5}) implies
\begin{equation}\label{ff6}
\frac{d{\mathcal F}_\gamma(X_\epsilon)}{d\epsilon}\Big|_{\epsilon=0}
=
- \int_{M} n\Lambda\psi\;dA
-\oint_{\partial {M}} \langle \delta X, R(p(\tilde{\xi}))\rangle\;d{\tilde s},
\end{equation}
which completes the proof of Proposition \ref{FV}.
\hfill $\Box$


\appendix
\section{Computations for Example \ref{ex22}} \label{proof:ex22}

Set
\begin{eqnarray}
\gamma_m(\cos\theta, \sin\theta)&=&(\cos^{2m}\theta+\sin^{2m}\theta)^{1/(2m)}\\
&=&\gamma_m(\theta).
\end{eqnarray}
The homogeneous extension of $\gamma$ is
\begin{equation}
\overline{\gamma}_m(x_1, x_2)=
(x_1^{2m}+x_2^{2m})^{1/(2m)}.
\end{equation}
And so, 
\begin{equation}
D\overline{\gamma}_m(x_1, x_2):=
((\overline{\gamma}_m)_{x_1}, (\overline{\gamma}_m)_{x_2})
=
(x_1^{2m}+x_2^{2m})^{(1/(2m))-1}(x_1^{2m-1}, x_2^{2m-1})
.
\end{equation}
Hence, the Cahn-Hoffman map $\xi_m$ for $\gamma_m$ is computed as follows.
\begin{eqnarray}
\xi_m(\cos\theta, \sin\theta)&:=& (f_m(\theta), g_m(\theta))
\\
&:=&D_{S^1}\gamma_m\Big|_{\nu}+\gamma_m(\nu)\nu
=D\overline{\gamma}_m(x_1, x_2)|_{(\cos\theta, \sin\theta)}\\
&=&
(\cos^{2m}\theta+\sin^{2m}\theta)^{(1/(2m))-1}(\cos^{2m-1}\theta, \sin^{2m-1}\theta). \label{kusi}
\end{eqnarray}
Also, we have
\begin{eqnarray}
\Delta\overline{\gamma}_m(x_1, x_2)
&:=&
(\overline{\gamma}_m)_{x_1x_1}
+
(\overline{\gamma}_m)_{x_2x_2} \nonumber\\
&=&(2m-1)x_1^{2m-2}x_2^{2m-2}(x_1^2+x_2^2)(x_1^{2m}+x_2^{2m})^{(1/(2m))-2},
\end{eqnarray}
\begin{eqnarray}
A_m:=d\xi_m &:=&
D^2\gamma_m +\gamma_m \cdot 1
=\Delta\overline{\gamma}_m(x_1, x_2)\Big|_{(\cos\theta, \sin\theta)}
\nonumber\\
&=&(2m-1)\cos^{2m-2}\theta \sin^{2m-2}\theta(\cos^{2m}\theta+\sin^{2m}\theta)^{(1/(2m))-2}.
\end{eqnarray}
Hence, 

(i) If $m \ge 2$, $A_m=0$ at $\theta=(1/2)\ell \pi$, ($\ell\in \mathbb Z$).

(ii) $A_m$ is positive-definite on $S^1\setminus\{(\cos\theta, \sin\theta) \:|\:
\theta=(1/2)\ell \pi, (\ell\in \mathbb Z)\}$. 
\newline Using (\ref{kusi}), we compute the curvature $\kappa_m$ of $\xi_m$ with respect to the outward-pointing normal $\nu$ and obtain the followings. 
\begin{equation}\label{def-k}
\kappa_m =
\frac{-f_m'g_m''+f_m''g_m'}{((f_m')^2+(g_m')^2)^{3/2}},
\end{equation}
\begin{equation}\label{element}
((f_m')^2+(g_m')^2)^{1/2}
=
(2m-1)\cos^{2m-2}\theta\sin^{2m-2}\theta(\cos^{2m}\theta+\sin^{2m}\theta)^{\frac{1}{2m}-2},
\end{equation}
\begin{equation}\label{numerator}
-f_m'g_m''+f_m''g_m'
=
-(2m-1)^2\cos^{4m-4}\theta\sin^{4m-4}\theta(\cos^{2m}\theta+\sin^{2m}\theta)^{\frac{1}{m}-4}
\end{equation}
Using (\ref{def-k}), (\ref{element}), and (\ref{numerator}), we obtain
\begin{equation}\label{kap}
\kappa_m=\frac{-1}{2m-1}\cos^{-2m+2}\theta\sin^{-2m+2}\theta(\cos^{2m}\theta+\sin^{2m}\theta)^{2-\frac{1}{2m}}.
\end{equation}
Hence, for any $\ell \in {\mathbb Z}$ and $m \ge 2$, 
\begin{equation}
\lim_{\theta \to \frac{\ell}{2}\pi}\kappa_m(\theta)=-\infty
\end{equation}
holds. 
On the other hand, using (\ref{element}) and (\ref{kap}), we obtain
\begin{equation}
\kappa_m \;ds=
\frac{-f_m'g_m''+f_m''g_m'}{((f_m')^2+(g_m')^2)^{3/2}}((f_m')^2+(g_m')^2)^{1/2}\;d\theta
=-d\theta.
\end{equation}
Hence 
\begin{equation}\label{tot1}
\int_0^{2\pi}\kappa_m\;ds
=-\int_0^{2\pi}d\theta
=-2\pi.
\end{equation}
holds. However, (\ref{tot1}) is trivial because $\xi:S^1\to{\mathbb R}^2$ is a front with $\nu \in S^1$.

\section{Proof of Lemma \ref{HL} (i)} \label{pfd}

By the same way as that outlined in \cite[p.699]{HL2008}, we can prove the desired result as follows. 
Let $X:M^n \to {\mathbb R}^{n+1}$ be an immersion with unit normal $\nu:M\to S^n$. 
Let $(u_1, \cdots, u_n)$ be local coordinates in $M$. 
We use the same notations as in \S \ref{a-curv}. 
Note 
$$
S^\gamma=
-d\tilde{\xi}
=-d\xi\circ d\nu
=A({h}_{ij})(g^{ij}), \quad A:=D^2\gamma+\gamma\cdot 1.
$$
By choosing a suitable coordinate system of $M$, $Q :=({h}_{ij})(g^{ij})$ can be a symmetric matrix.

$A$ can be represented as an $n\times n$ matrix. 
Since $A$ is symmetric, it has $n$ real eigenvalues, which we denote by $\lambda_1, \cdots, \lambda_n$. 
Since $A$ is positive definite, all of $\lambda_1, \cdots, \lambda_n$ are positive. 

Since $A$ is symmetric, it has an eigendecomposition $PD\;{}^t\!P$, where $P$ is a real orthogonal matrix whose rows comprise an orthonormal basis of eigenvectors of $A$, and $D$ is a real diagonal matrix whose main diagonal contains the corresponding eigenvalues. Hence, we can write
\[
  D = \left(
    \begin{array}{ccccc}
      \lambda_1 & 0 & \cdots& \cdots & 0\\
      0 & \lambda_2 & 0 & \cdot & \cdot \\
      \cdot & 0 & \cdot & 0 & \cdot \\
      \cdot & \cdot & 0 & \lambda_{n-1} & 0 \\
      0 & \cdots &\cdots & 0 & \lambda_n
    \end{array}
  \right)=(\lambda_i\delta_{ij})_{i, j=1, \cdots, n}.
\]
Set
\[
  \tilde{D} = \left(
    \begin{array}{ccccc}
      \sqrt{\lambda_1} & 0 & \cdots& \cdots & 0\\
      0 &  \sqrt{\lambda_2} & 0 & \cdot & \cdot \\
      \cdot & 0 & \cdot & 0 & \cdot \\
      \cdot & \cdot & 0 &  \sqrt{\lambda_{n-1}} & 0 \\
      0 & \cdots &\cdots & 0 & \sqrt{\lambda_n}
    \end{array}
  \right)=(\sqrt{\lambda_i}\delta_{ij})_{i, j=1, \cdots, n}.
\]
Then,
$$
A=PD\;{}^t\!P
=P\tilde{D}\tilde{D}\;{}^t\!P
=P\tilde{D}\;{}^t\!\tilde{D}\;{}^t\!P
=P\tilde{D}\;{}^t\!(P\tilde{D}).
$$
Set
$$
\tilde{P}:=P\tilde{D}.
$$
Then, 
$$
A=\tilde{P}\;{}^t\!\tilde{P}.
$$
Note that $\tilde{P}$ is a regular matrix. Hence, we have 
\begin{eqnarray}
\det(tI-S^\gamma)&=&\det(tI-AQ)=\det(tI-\tilde{P}\;{}^t\!\tilde{P}Q)\\
&=&\det(\tilde{P}^{-1}(tI-\tilde{P}\;{}^t\!\tilde{P}Q)\tilde{P})
=\det(tI-\;{}^t\!\tilde{P}Q\tilde{P}).
\end{eqnarray}
Hence, $S^\gamma$ has the same eigenvalues as those of ${}^t\!\tilde{P}Q\tilde{P}$. 
Since $Q$ is symmetric, ${}^t\!\tilde{P}Q\tilde{P}$ is also symmetric. Therefore, $S^\gamma$ has $n$ real eigenvalues. 

 \begin{flushleft}
Miyuki K{\footnotesize OISO} \\
Institute of Mathematics for Industry  \\
Kyushu University  \\
Motooka Nishi-ku, Fukuoka \\
FUKUOKA 819-0395,
JAPAN \\
E-mail: {\tt koiso@math.kyushu-u.ac.jp} \\
 \end{flushleft}

\end{document}